\documentclass[11pt]{article}

\usepackage[left=30mm,right=30mm, top=30mm, bottom=30mm]{geometry}
\usepackage{graphicx,color}
\usepackage[utf8]{inputenc}
\usepackage{amsmath}
\usepackage{amssymb}
\usepackage{amsthm}
\usepackage{comment}
\usepackage{enumitem}
\setlist[enumerate]{parsep=0pt}
\usepackage{mathtools}
\usepackage{cases}
\usepackage{multicol}
\usepackage{tabularx}
\usepackage{authblk}
\usepackage[nocompress]{cite}
\usepackage{todonotes}
\usepackage[colorlinks]{hyperref}
\hypersetup{
	colorlinks=true,
	linkcolor=blue,
	citecolor=red
}

\usepackage{tikz}
\usepackage{tikz-cd}
\usetikzlibrary{%
	matrix,%
	calc,%
	arrows%
}
\usetikzlibrary{arrows}

\renewcommand{\Re}{\operatorname{Re}}

\theoremstyle{plain}
\newtheorem{thm}{Theorem}[section]
\newtheorem{prop}[thm]{Proposition}
\newtheorem{lem}[thm]{Lemma}
\newtheorem{cor}[thm]{Corollary}

\theoremstyle{definition}
\newtheorem{defi}[thm]{Definition}
\newtheorem{ex}[thm]{Example}
\newtheorem{assumption}[thm]{Assumption}

\theoremstyle{remark}
\newtheorem{remark}[thm]{Remark}

\newcommand{\C}{\mathbb C}
\newcommand{\R}{\mathbb R}
\newcommand{\N}{\mathbb N}

\newcommand{\dom}{\mathrm{dom}}

\renewcommand{\i}{\operatorname{i}}

\newcommand{\e}{\mathrm{e}}
\newcommand{\dx}[1][x]{\,\mathrm{d}#1}
\newcommand{\dd}{\mathrm{d}}

\newcommand{\Real}{\operatorname{Re}}

\newcommand{\ran}{
\mathchoice%
    {\rm ran\,}
    {\rm ran\,}
    {\rm ran}
    {\rm ran}
}
\renewcommand{\ker}{
\mathchoice%
    {\rm ker\,}
    {\rm ker\,}
    {\rm ker}
    {\rm ker}
}

\newcommand{\inv}{^{-1}}
\newcommand{\ainv}{^{-\ast}}

\title{Integrated Semigroups for abstract differential-algebraic equations}

\author{Mehmet Erbay, Birgit Jacob and Timo Reis}

\date{\today}

\begin{document}
\maketitle

\begin{abstract}
We study integrated semigroups for infinite-dimensional differential-algebraic equations (DAEs) admitting a resolvent index. Building on the notion of integrated semigroups for the abstract Cauchy problem $\frac{\dd}{\dd t}x=Ax$, we extend this concept to the DAE setting. The resulting framework is used to analyze inhomogeneous DAEs and plays a central role in characterizing their solutions.
\end{abstract}

\section{Introduction} 
We consider infinite-dimensional differential-algebraic equations (short DAEs) 
\begin{equation}\label{eq:dae-introduction}
    \begin{aligned}
        \tfrac{\dd}{\dd t} Ex(t) &= Ax(t) + f(t), \quad t\in [0,\infty),\\
        x(0) &= x_0.
    \end{aligned}
\end{equation}
Here, $X$, $Z$ denote complex Banach spaces, $E$ is bounded from $X$ to $Z$, denoted by $E\in L(X,Z)$, $A\colon \dom(A)\subseteq X\to Z$ is closed and densely defined, $x_0\in X$ and $f\colon [0,\infty)\to Z$. Such DAEs arise in a variety of fields, such as boundary-coupled partial differential equations with finite-dimensional DAEs, of partial differential equations with algebraic constraints \cite{erbay_index_2024,gernandt_pseudo-resolvent_2023,reis_controllability_2008,reis_frequency_2005,thesis-reis,HankeReis2009HeatExchanger,trostorff_higher_2018}. These constraints impose additional restrictions on the dynamics and, in most cases, endow the operator 
$E$ with a nontrivial kernel, thereby complicating the search for solutions. 

In general, a solution of \eqref{eq:dae-introduction} can be explicitly constructed, provided the inhomogeneity is sufficiently smooth and the DAE satisfies appropriate regularity assumptions; that is, if the resolvent $(\lambda E-A)\inv$, for $\lambda \in \rho(E,A)\coloneqq \{\lambda \in \C \, \vert \, (\lambda E-A)\inv \in L(Z,X)\}$, grows with respect to $\lambda$ at most polynomially on a right-half plane of the complex numbers \cite{reis_frequency_2005}. 
The drawback is that, although existence is ensured, the solution must be obtained via an inverse Laplace transform, which is often a substantial obstacle. It is therefore desirable to adopt a time-domain approach.
This is the point where integrated semigroups come into play: A linear operator $A\colon \dom(A)\subseteq X\to X$ generates an $n$-times integrated semigroup (for an $n\in \N$), if
\begin{equation}\label{eq:growth-rate-A}
    \Vert (\lambda I_X-A)\inv\Vert \leq C \vert \lambda \vert^{n-1}, \quad \lambda \in \C_{\Real >\omega},
\end{equation}
for some $C>0$, $\omega\in \R$ and $\C_{\Re >\omega}=\{\lambda \in \C \, \vert \, \Re \lambda >\omega\}$ \cite[Prop.~2.3 \& Thm.~4.2]{neubrander_integrated_1988}. (For $n=0$ and $C=1$ this yields that $A$ is the generator of an analytic semigroup.) The integrated semigroup is then defined on the closure of $\ran((\lambda I_X-A)\inv)^{n+2}$, which is equal to $X$ if $A$ is densely defined \cite{oharu71}.  
Further, it is possible to show that for every initial value of $\dom(A^{n+2}) = \ran((\lambda I_X-A)\inv)^{n+2}$ the Cauchy problem $\frac{\dd}{\dd t}x = Ax$ has a classical solution, which can be described with the help of such an integrated semigroup. For further information regarding integrated semigroups for the classical Cauchy problem we recommend \cite{neubrander_integrated_1988, arendt_vector-valued_1987, arendt_vector-valued_2011, engel_one-parameter_2000, thiemea_integrated_1990}.

Note that the growth condition \eqref{eq:growth-rate-A} coincides with the \textit{complex resolvent index}, which is in \cite{erbay_index_2024}  defined for the DAE case. 
In fact, the index of a DAE can be defined in various ways, such as the {\em differentiation index} \cite{campbell2, campbell_index_1995, Griepentrog1992Toward, kunkel_differential-algebraic_2006, Lucht1999IndexesAS}, the {\em perturbation index} \cite{campbell_index_1995, mehrmann_index_2015}, the {\em nilpotency index} \cite{kunkel_differential-algebraic_2006}, the {\em geometric index} \cite{kunkel_differential-algebraic_2006, mehrmann_index_2015, rabier_theoretical_2002, reich_geometrical_1990}, the {\em radiality index} \cite{sviridyuk_linear_2003, jacob_solvability_2022} and the aforementioned {\em resolvent index}  \cite{gernandt_pseudo-resolvent_2023, gernandt_linear_2021, sviridyuk_linear_2003, trostorff_semigroups_2020, trostorff_higher_2018}.
Note, that in infinite dimensions all these index notions do not have to be equivalent (see \cite{erbay_index_2024}).

This raises the question to what extent the concept of integrated semigroups can be extended to DAEs, which we dealt with in this article. 
So far the generation of integrated semigroups for DAEs have been analysed in \cite{melnikova_abstract_nodate, melnikova_properties_1996}, where they studied the well-posedness of such systems for $n=1$. 
We now build upon that work and extend it by considering DAEs with inhomogeneities, which is substantial for future research and applications, such as DAEs with input and output (in particular port-Hamiltonian DAEs). 
While doing so, we make use of already known concepts and theories concerning classical integrated semigroups \cite[Ch.~3.2]{arendt_vector-valued_2011}. 

We subsequently focus on the study of \eqref{eq:dae-introduction} using pseudo-resolvents, as was done in \cite{gernandt_pseudo-resolvent_2023}. To be more precise, instead of studying the DAE \eqref{eq:dae-introduction} directly, it is possible to study for $\lambda \in \rho(E,A)$ the altered DAE
\begin{align}\label{eq:pseudo-dae-0}
    \tfrac{\dd}{\dd t}R(\lambda) w_\lambda(t) = w_\lambda(t) + f_\lambda(t),
\end{align}
where $R(\lambda)$ either denotes the \textit{left-$E$-resolvent} or the \textit{right-$E$-resolvent} of $(E,A)$. These are respectively given by the functions $\lambda\mapsto E(\lambda E-A)\inv$, $\lambda\mapsto (\lambda E-A)\inv E$ defined on the resolvent set of $(E,A)$, and with adapted initial condition and inhomogeneity (both depending on $\lambda$). 
In \cite{gernandt_pseudo-resolvent_2023}, it was shown—under an index concept slightly stronger than the resolvent index—that \(R(\lambda)\) is a resolvent on the closed subspace \(\overline{\operatorname{ran} R(\lambda)^k}\). Specifically, there exists an operator
\begin{equation*}
    A_R \colon \dom(A_R)\subseteq \overline{\operatorname{ran} R(\lambda)^k}\to \overline{\operatorname{ran} R(\lambda)^k}
\end{equation*}
such that
\begin{equation*}
    R(\lambda)\big|_{\overline{\operatorname{ran} R(\lambda)^k}} = (\lambda I - A_R)^{-1}.
\end{equation*}
Further, they have shown that this operator $A_R$ contributes directly to solving \eqref{eq:pseudo-dae-0} by generating a strongly continuous semigroup, which extends to a degenerate semigroup on the whole space. More details regarding this subject can be found in \cite{krejn_linear_1971, melnikova_abstract_nodate, favini_degenerate_1998, showalter_monotone_1997}.
In the following, we take on the same task under slightly weaker assumptions and show that $A_R$ always generates a strongly continuous semigroup on a subspace.

The paper is organized as follows. In Section \ref{section:solutions} we introduce the concept of integrated semigroups for DAEs and investigate several of their properties. 
In Section \ref{section:inhomogeneous-daes} we analyse the representations of solutions of inhomogeneous DAEs imposing various conditions on the inhomogeneity. To be more precise, we investigate the cases where the inhomogeneity maps into the images or the kernels of the right- or left-$E$-resolvent, which is often the case for a DAE with Weierstraß form. Subsequently we consider an example which demonstrates that these assumptions are generally too restrictive and provide an approach in Hilbert spaces in which these solutions can be characterized for arbitrary inhomogeneities.
In Section \ref{section:semigroup-generation} we conclude our research on integrated semigroups by showing, when they generate a strongly continuous semigroup on a subspace, as is usually the case for ordinary partial differential equations. In particular, for DAEs, this is the case when certain conditions are imposed on the kernel of $E$.

\textbf{Notation.} If not mentioned otherwise, $X$ and $Z$ denote complex Banach spaces. The norm in $X$ and $Z$ will be denoted by $\Vert \cdot\Vert_X$ and $\Vert\cdot\Vert_Z$, or simply $\Vert \cdot\Vert$, if it is clear from context. Similarly, we denote by $I_X$ and $I_Z$ the identity map in $X$ and $Z$ and write only $I$, if it is clear from context. 
For $\omega \in \R$ we set $\C_{\Real \geq \omega}\coloneqq \{\lambda \in \C \,\vert\, \Real \lambda \geq \omega\}$. 

For an interval $J\subseteq\R$ and a Banach space $X$, we write $C(J;X)$ for the space of continuous $X$-valued functions and $C^{p}(J;X)$ (with $p\in\N$) for the space of $p$-times continuously differentiable maps. Moreover, $L^{p}(J;X)$ ($1\le p\le\infty$) denotes the Lebesgue–Bochner space of $p$-integrable $X$-valued functions (all integrals are understood in the Bochner sense \cite{Diestel77}), and $H^{k}(J;X)\subseteq L^{2}(J;X)$ is the Sobolev space of $X$-valued functions whose weak derivatives up to order $k$ belong to $L^{2}(J;X)$.
By $L(X,Z)$ we denote the space of bounded linear operators mapping from $X$ to $Z$ and we abbreviate $L(X)\coloneqq L(X,X)$. 
By $\dom(A)$ we denote the domain of a (not necessary bounded) linear operator $A\colon\dom(A)\subseteq X\to Z$. In this context, $A$ is called \textit{densely defined}, if $\dom(A)$ is dense in $X$, and \textit{closed}, if the graph of $A$ 
is a closed subspace of $X\times Z$.

Let $E\in L(X,Z)$ and $A\colon \dom(A)\subseteq X\to Z$ be closed and densely defined. We call $(\lambda I-A)\inv $ \textit{the resolvent of $A$} for all $\lambda \in \rho(A)\coloneqq \{\lambda \in \C \, \vert \, (\lambda I-A) \text{ boundedly invertible}\}$. Further, we call $(\lambda E-A)\inv$ \textit{the (generalised) resolvent of $(E,A)$} for all $\lambda \in \rho(E,A)\coloneqq \{ \lambda \in \C \, \vert \, (\lambda E-A)\inv \text{ boundedly invertible}\}$.
For notational convenience, we define the \textit{right} and \textit{left resolvents} of $(E,A)$ by
\begin{equation*}
    R_r(\lambda) \coloneqq (\lambda E - A)^{-1}E,
    \qquad
    R_l(\lambda) \coloneqq E(\lambda E - A)^{-1}.
\end{equation*}
In the case when $X=Z$ is a Hilbert space, we call $E\in L(X)$ \textit{positive}, if $\langle Ex,x\rangle\geq 0$ for all $x\in X$ and we call $A\colon\dom(A)\subseteq X\to X$ \textit{dissipative}, if $\Real \langle Ax,x\rangle\leq 0$ for all $x\in \dom(A)$.

\section{Differential-algebraic integrated semigroups} \label{section:solutions}
\subsection{Preliminaries}\label{sec:prelim}
We start by focusing on the DAE 
\begin{equation}\label{eq:dae-0}
    \begin{aligned}
        \tfrac{\dd}{\dd t} Ex(t) &= Ax(t) + f(t), \quad t\in [0,\infty),\\
        x(0) &= x_0,
    \end{aligned}
\end{equation}
with $x_0 \in X$ and $f\colon [0,\infty)\to Z$. In this context, the operator pair $(E,A)$ has the following properties.

\begin{assumption}\label{assumption:1}\hfill
    \begin{enumerate}[label=({\alph*)}]
        \item $X$ and $Z$ are Banach spaces.
        \item $E$ is a bounded linear operator from $X$ to $Z$.
        \item $A\colon\dom(A)\subseteq X\to Z$ is closed and densely defined.
        \item The operator pair $(E,A)$ has a \textit{complex resolvent index}, i.e.~the smallest number $p_{\mathrm{res}}^{(E,A)} \in \N_0$, such that there exists a $\omega \in\R$, $C>0$ with $\C_{\Re>\omega}\subseteq \rho(E,A)$ and 
        \begin{equation}\label{eq:complex-resolvent-index}
            \Vert (\lambda E-A)^{-1} \Vert \leq C\vert\lambda\vert^{p_{\mathrm{res}}^{(E,A)}-1}\quad \text{for all }\lambda\in\C_{\Re>\omega}.
        \end{equation}     
    \end{enumerate}
\end{assumption}


\begin{defi}[Solutions]\hfill
    \begin{enumerate}[label=({\roman*)}]
        \item We call $x\colon [0,\infty) \to X$ a \textit{classical solution} of \eqref{eq:dae-0}, if $x\in C([0,\infty); X)$, $Ex\in C^1([0,\infty); Z)$, $x(t)\in \dom(A)$, $t\geq 0$ and $x$ solves \eqref{eq:dae-0}.
        \item We call $x\colon [0,\infty) \to X$ a \textit{mild solution} of \eqref{eq:dae-0}, if $x\in L^1([0,\infty); X)$, $Ex\in C([0,\infty); Z)$, $f\in L^1([0,\infty); Z)$, $\int_0^tx(\tau)\dx[\tau]\in \dom(A)$ and
        \begin{equation}\label{eq:mild-sol}
            Ex(t) - Ex(0) = A\int_0^t x(\tau)\dx[\tau] + \int_0^t f(\tau)\dx[\tau].
        \end{equation}
    \end{enumerate}
\end{defi}

The following theorem was formulated under the assumption that $X$ and $Z$ are complex Hilbert spaces. After examining its proof, it becomes clear that the result also holds in the Banach space case.

\begin{thm}\cite[Thm.~2.2 \& Rem.~2.3]{erbay_jacob_morris24}\label{thm:solutions-complex-resolvent-index}\hfill\\
    Let $(E,A)$ satisfy Assumption \ref{assumption:1}, $f=0$ and $\mu \in \rho(E,A)$.
    \begin{enumerate}[label=({\alph*)}]
        \item If $x_0 \in \ran R_r(\mu)^{p_{\mathrm{res}}^{(E,A)}+1}$, then there exist a unique mild solution $x$ of \eqref{eq:dae-0}.
        \item If $x_0 \in \ran R_r(\mu)^{p_{\mathrm{res}}^{(E,A)}+2}$, then there exist a unique classical solution $x$ of \eqref{eq:dae-0}.
    \end{enumerate}
In both cases, let $\omega\in\R$ and $C>0$ be such that \eqref{eq:complex-resolvent-index} holds with $p=p_{\mathrm{res}}$. Then, for every $\omega_0>\omega$, the solution is given by
    \begin{equation*}
        x(t)=\frac{(-1)^p}{2\pi \i} \int_{\omega_0-\i \infty}^{\omega_0+\i \infty} \e^{\lambda t} \frac{R_r(\lambda)z_0}{(\lambda-\mu)^p}\dx[\lambda], \quad t\geq 0,
    \end{equation*}
where $\mu \in \rho(E,A)\cap\C_{\Re>\omega_0}$ and $z_0\in X$ satisfies $x_0=R_r(\mu)^p z_0$, with $p=p_{\mathrm{res}}^{(E,A)}+1$ for the mild solution and $p=p_{\mathrm{res}}^{(E,A)}+2$ for the classical solution, respectively.\end{thm}

\begin{remark}
    The proof of Theorem \ref{thm:solutions-complex-resolvent-index} requires that the right-resolvent grows at most polynomially and, thus, the assumption that $p^{(E,A)}_{\mathrm{res}}$ exists can be weakened to
    \begin{equation}\label{eq:complex-pseudoresolvent-index}
        \Vert (\lambda E-A)\inv E\Vert \leq C \vert \lambda \vert^{p-1}, \quad \lambda \in C_{\Re >\omega},
    \end{equation}
    for a $\omega \in \R$ and $p\in \N$.
\end{remark}

Let $p\in \N$ and let $H^p([0,\infty);Z)$ 
be the Sobolev space of $Z$-valued functions with $p$ weak derivatives in $L^2$.
Consider the space     \begin{equation}\label{eq:H0}
        H^p_{0,l}([0,\infty); Z)\coloneqq \{ f \in H^p([0,\infty); Z) \, \vert \, f(0)=\ldots= f^{(p-1)}(0) = 0\}.
    \end{equation}
\begin{thm}\cite[Thm.~2.4 \& 2.6]{thesis-reis}\label{thm:inhomog-solutions-complex-resolvent-index}\hfill\\
    Let $(E,A)$ satisfy Assumption \ref{assumption:1}.
    \begin{enumerate}[label=({\alph*)}]
        \item If $f\in H_{0,l}^{p_{\mathrm{res}}^{(E,A)}}([0,\infty); Z)$, then there exist a unique mild solution of \eqref{eq:dae-0} with $x_0=0$.
            \item If $f\in H_{0,l}^{p_{\mathrm{res}}^{(E,A)}+1}([0,\infty); Z)$, then there exist a unique classical solution of \eqref{eq:dae-0}  with $x_0=0$.
    \end{enumerate}
\end{thm}

\subsection{Construction of differential-algebraic integrated semigroups}
Let $X$, $Z$, $E$, and $A$ be as in Assumption \ref{assumption:1}. 
Our goal is to develop the solution theory of DAEs and to present it using different approaches. In particular, we extend the theory of integrated semigroups in order to represent the resolvent of $(E,A)$ in a Laplace-transform-like form, which is possible whenever the complex resolvent index exists.

We start by considering the homogeneous system
\begin{equation}\label{eq:dae}
    \begin{aligned}
        \tfrac{\dd}{\dd t} Ex(t) &= Ax(t), \quad t\in [0,\infty),\\
        x(0) &= x_0.
\end{aligned}
    \end{equation}
For the standard abstract Cauchy problem (i.e.~with $X=Z$ and $E=I_X$), the solutions can be represented by a strongly continuous semigroup, provided that the operator $A$ generates such a semigroup. If this is not the case, then, provided the resolvent of $A$ is polynomially bounded on some right half-plane, the solutions can be described by an integrated semigroup, which in turn yields representation formulas for the resolvent of $A$ (see \cite{neubrander_integrated_1988,arendt_vector-valued_1987,arendt_vector-valued_2011}). 
To be more precise, as long as the resolvent of $A$ grows polynomially, it is possible to show that for every initial value in $\dom(A^k)=\ran ((\lambda I_X-A)\inv)^k$ (for a $\lambda \in \rho(A)$) there is a classical solution of the Cauchy problem (where $k\in \N$ depends on the growth rate of the resolvent) (see \cite{krejn_linear_1971, BEALS1972281}). In this case, $A$ is the generator of an integrated semigroup, i.e.~for all $\lambda \in \rho(A)\cap \C_{\Re >\omega}$ (for $\omega>0$ large enough)
\begin{equation*}
    (\lambda I_X-A)\inv x = \lambda^p \int_0^\infty \e^{-\lambda t} S(t) x \dx[t], \quad x\in X,
\end{equation*}
holds, where $(S(t))_{t\geq 0}$ is a family of bounded operators on $X$. 

Now, we go back to our original problem \eqref{eq:dae}. Since we assumed the existence of the complex resolvent index throughout this section, one always finds a unique (mild) solution of \eqref{eq:dae} for all $x_0 \in \ran R_r(\lambda)^p$ by Theorem \ref{thm:solutions-complex-resolvent-index} (with $\lambda \in \rho(E,A)$). The difference to the Cauchy problem is that these initial values are not dense in the whole space $X$. In general, under our given assumptions it is not even possible to split the space into $\overline{\ran R_r(\lambda)^p}\oplus \ker R_r(\lambda)^p$, as it is possible in finite-dimensions thanks to the Weierstraß form (see \cite{kunkel_differential-algebraic_2006}). However, this is not a problem as the solutions of \eqref{eq:dae} always lie in $\overline{\ran R_r(\lambda)^{p}}$ (see \cite[Prop.~4.1\&Prop.~5.1]{trostorff_semigroups_2020}).

\begin{thm}\label{thm:1}\hfill\\
    Let $(E,A)$ satisfy Assumption \ref{assumption:1} and set $p:=p_{\mathrm{res}}^{(E,A)}+1$. Then, there exists a family of operators $(S_r(t))_{t\geq 0}$ on $X_{\ran} \coloneqq \overline{\ran R_r(\lambda)^{p}}$ 
    and $\omega\in \R$, $C>0$ with $\Vert S_r(t)\Vert \leq C\e^{\omega t}$ and
    \begin{equation}\label{eq:rechts-res}
        R_r(\lambda)x_0 = \lambda^p \int_0^\infty \e^{-\lambda t} S_r(t)x_0 \dx[t], \quad \text{ for all $\lambda \in \rho(E,A)\cap \C_{\Re >\omega}$ and $x_0\in X_{\ran}$.}
    \end{equation}
    
\end{thm}

\begin{proof}
    First assume that $x_0 \in \ran R_r(\mu)^p$. Then there exists some 
    $z_0\in X$ with $x_0 = R_r(\mu)^p z_0$
    By Theorem \ref{thm:solutions-complex-resolvent-index} there exists a mild solution $x$ of \eqref{eq:dae}. The solution formula in that theorem implies that, for some $\omega_0\in\R$,
    \begin{equation*}
        \Vert x(t) \Vert \leq \frac{1}{2\pi} \int_{\omega_0 -\i \infty}^{\omega_0 + \i \infty} \e^{\omega_0 t} \frac{\Vert R_r(\lambda) z_0\Vert}{\vert \lambda -\mu\vert^p}\dx[\lambda] \leq C_0 \e^{\omega_0 t} \Vert z_0\Vert\; \text{ for all $t\ge0$},
    \end{equation*}
    for some suitable $C_0>0$. 
    Using $\mu R_r(\mu) - (\mu E-A)\inv A=I_{\dom(A)}$ we obtain 
    \begin{align*}
        x_1(t) \coloneqq \int_0^t x(t_0) \dx[t_0] &= \left(\mu R_r(\mu) - (\mu E-A)\inv A\right) \int_0^t x(t_0) \dx[t_0] \\ 
        &= \mu\int_0^t R_r(\mu)x(t_0) \dx[t_0] - R_r(\mu)x(t) + R_r(\mu)x(0)\; \text{ for all $t\ge0$}.
    \end{align*}
    By \cite[Lem.~8.1]{gernandt_pseudo-resolvent_2023} $R_r(\mu)x$ is a mild solution of \eqref{eq:dae} with initial value $R_r(\mu)x_0$ and, thus, for some $\omega_1\in \R$ and $C_1 >0$,
    \begin{equation*}
        \Vert x_1(t)\Vert \leq C_1\e^{\omega_1 t}  \Vert R_r(\mu)z_0\Vert\; \text{ for all $t\ge0$}.
    \end{equation*}
    Inductively, we define 
\begin{equation}\label{eq:x_k}
        x_k(t) \coloneqq \int_0^t \frac{(t-s)^{k-1}}{(k-1)!} x(\tau)\dx[\tau] = \int_0^t \left( \int_0^{t_{p-1}} \ldots \left( \int_0^{t_{1}} x(t_0) \dx[t_0]\right) \ldots \dx[t_{p-2}] \right) \dx[t_{p-1}],
    \end{equation} 
    and we obtain, for some $k\in \{1,\ldots,p\}$ and some $\omega_k \in \R$, $C_k>0$,    \begin{equation}\label{eq:a}
        \Vert x_k(t)\Vert \leq C_k \e^{\omega_k t} \Vert R_r(\mu)^kz_0\Vert  \leq C_k \e^{\omega_k t} \Vert x_0 \Vert.
    \end{equation}
     Thus, for every $t\geq 0$ the mapping
    \begin{equation*}
        S_r(t)\colon \ran R_r(\mu)^p \to X_{\ran}, \quad x_0 \mapsto x_p(t)
    \end{equation*}
    is bounded with
    \begin{equation}\label{eq:estimate-1}
        \Vert S_r(t)x_0\Vert \leq C_p \e^{\omega_p t} \Vert R_r(\mu)^p z_0\Vert =e^{\omega t} \Vert x_0\Vert,
    \end{equation}
    for an appropriate $\omega\in \R$. By \cite[Prop.~4.1 \& Prop.~5.1]{trostorff_semigroups_2020} every mild solution maps into $X_{\ran}$. Therefore, $S_r(t)$ is indeed well-defined.
    Since $\ran R_r(\mu)^p$ is obviously dense in $X_{\ran}$, we extend $S_r(t)$ to $X_{\ran}$ continuously and denote it again by $S_r(t)$.
    
    As $x$ is a mild solution of \eqref{eq:dae} one has $x_1(t)\in \dom(A)$ and $Ax_1(t) = Ex(t)-Ex_0$. Hence, $x_1$ and $Ax_1$ are both integrable and, since $A$ is closed, $x_2(t)\in \dom(A)$, $t\geq 0$, and $Ax_2(t) = A\int_0^t x_1(\tau)\dx[\tau] = \int_0^t Ax_1(\tau)\dx[\tau]$. By integrating $Ax_1 = Ex-Ex_0$ and using again the closedness of $A$ one shows $x_3(t)\in \dom(A)$, $t\geq 0$, and $Ax_3(t) = \int_0^t Ax_2(\tau)\dx[\tau]$. Inductively, $S_r(t)x_0 \in \dom(A)$, $t\geq 0$ and, by integration by parts,
    \begin{align*}
        AS_r(t) x_0 &= \int_0^t \frac{(t-\tau)^{p-2}}{(p-2)!} \underbrace{A\int_0^\tau x(\sigma)\dx[\sigma]}_{=Ex(\tau) - Ex_0}\\
        &= E\underbrace{\int_0^t\frac{(t-\tau)^{p-2}}{(p-2)!}x(\tau)\dx[\tau]}_{=\frac{\dd}{\dd t} S_r(t)x_0} -\frac{t^{p-1}}{(p-1)!} Ex_0 \\
        &= E\left(\frac{\dd}{\dd t}S_r(t)x_0 - \frac{t^{p-1}}{(p-1)!}x_0\right).
    \end{align*}
Since $A$ is closed, we obtain $\lambda^p \int_0^T \e^{-\lambda t} S_r(t)x_0\dx[t] \in \dom(A)$, $T\geq 0$, and
    \begin{align*}
        A \Big(\lambda^p \int_0^T & \e^{-\lambda t} S_r(t)x_0\dx[t]\Big) \\ 
        &= \lambda^p \int_0^T \e^{-\lambda t} AS_r(t)x_0\dx[t] \\
        &= \lambda^p \int_0^T \e^{-\lambda t} E\left(\frac{\dd}{\dd t}S_r(t)x_0 - \frac{t^{p-1}}{(p-1)!}x_0\right) \dx[t] \\
        &= \lambda^p \e^{-\lambda T} ES_r(T)x_0 + \lambda E \lambda^{p} \int_0^T \e^{-\lambda t}S_r(t)x_0 \dx[t] - \lambda^p \int_0^T \e^{-\lambda t} \frac{t^{p-1}}{(p-1)!} Ex_0 \dx[t]\\
        &\stackrel{T\to \infty}{\to} 0 + \lambda E \tilde R(\lambda)x_0 + Ex_0,
    \end{align*}
    where
    \begin{equation*}
        \tilde R(\lambda)x_0\coloneqq \lambda^p \int_0^\infty \e^{-\lambda t} S_r(t)x, \dx[t], \quad x_0\in X_{\ran},
    \end{equation*}
    which is well defined and bounded by \eqref{eq:estimate-1}.
    Hereby, we have used $S_r(0)x_0 = 0$ and integration by parts. Again, by using the closedness of $A$, we obtain $\tilde R(\lambda)x_0 \in \dom(A)$ and $A\tilde R(\lambda)x_0 = \lambda E\tilde R(\lambda)x_0 - Ex_0$. This is equivalent to $R_r(\lambda)x_0 = \tilde R(\lambda)x_0$ for all $x_0 \in \ran R_r(\mu)^p$. By using the closedness of $A$ once again, we can extend this to all $x_0\in X_{\ran}$ and obtain 
    \begin{equation*}
        R_r(\lambda) x_0 = \tilde R(\lambda)x_0 = \lambda^p \int_0^\infty \e^{-\lambda t} S_r(t) x_0\dx[t], \quad x_0\in X_{\ran}. \qedhere
    \end{equation*}
\end{proof}

\begin{remark}\hfill
Assume the hypotheses of Theorem~\ref{thm:1}.    \begin{enumerate}[label=({\roman*)}]
        \item One has $S_r(0)=0$, as $S_r(0)$ is already $0$ on $\ran R_r(\mu)^p$.        
        \item It follows from the argumentation in the above proof that $S_r(t)x_0\in \dom(A)$ for all $x_0\in \ran R_r(\mu)^p$.
        \item If $\overline{\ran R_r(\mu)^p} \cap \ker R_r(\mu)^p = \{0\}$, then $R_r(\mu)$ is a resolvent for an operator $A_0$ on $\ran R_r(\mu)^p$ and one can directly make use of the theory of integrated semigroups as seen in \cite{neubrander_integrated_1988, arendt_vector-valued_2011}. This is the case, for instances, if $(E,A)$ has a radiality index \cite{sviridyuk_linear_2003, erbay_jacob_morris24} or if 
        \begin{equation*}
            \Vert R_r(\lambda)x_0\Vert \leq \frac{M}{\lambda - \omega} \Vert x_0 \Vert, \;\text{ for all $\lambda \in (\omega, \infty)$,\; $x_0 \in \ran R(\omega)^{k-1}$.}
        \end{equation*}
        The latter condition was intensively studied in \cite{gernandt_pseudo-resolvent_2023}. 
    \end{enumerate}
\end{remark}

The next step is to show that the left-resolvent also admits a representation as the right-resolvent does in \eqref{eq:rechts-res}.

\begin{cor}\label{cor:1}\hfill\\
    Let $(E,A)$ satisfy Assumption \ref{assumption:1} and set $p:=p_{\mathrm{res}}^{(E,A)}+1$. Then, there exists a family of operators $(S_l(t))_{t\geq 0}$ on $Z_{\ran}\coloneqq \overline{\ran R_l(\lambda)^{p}}$ 
    and $\tilde \omega \in \R$, $\tilde C>0$ with $\Vert S_l(t)\Vert \leq \tilde C \e^{\tilde \omega t}$ and 
    \begin{equation}
        R_l(\lambda)z = \lambda^p \int_0^\infty \e^{-\lambda t} S_l(t)z \dx[t]\; \text{ for all $\lambda \in \rho(E,A)\cap \C_{\Re >\tilde\omega}$, $ z\in Z_{\ran}$.}
    \end{equation}
     
\end{cor}

\begin{proof}
    Let $\mu\in \rho(E,A)$ and consider
    \begin{equation*}
        \frac{\dd}{\dd t} \underbrace{E(\mu E-A)\inv}_{\eqqcolon \tilde E} z(t) = \underbrace{A(\mu E-A)\inv}_{\eqqcolon \tilde A} z(t).
    \end{equation*}
    Then, for all $\lambda \in \rho(\tilde E,\tilde A)$ it holds that 
    \begin{equation*}
        (\lambda \tilde E-\tilde A)\inv = (\mu E-A)(\lambda E-A)\inv,
    \end{equation*}
    and therefore
    \begin{equation*}
        \tilde E (\lambda \tilde E-\tilde A)\inv = E(\lambda E-A)\inv = (\lambda\tilde E-\tilde A)\inv \tilde E.
    \end{equation*}
    Since $(E,A)$ has complex resolvent index $p_{\mathrm{res}}^{(E,A)}$ the same holds for $(\tilde E, \tilde A)$. Using Theorem \ref{thm:1} there exists a family $(S_l(t))_{t\geq 0}$ with $\Vert S_l(t)\Vert \leq \tilde C \e^{\tilde \omega t}$ (for an $\tilde \omega \in \R$ and $\tilde C>0$) with
    \begin{equation*}
        R_l(\lambda) z = (\lambda \tilde E-\tilde A)\inv \tilde E z = \lambda^p \int_0^\infty \e^{-\lambda t} S_l(t) z \dx[t], \quad 
        z\in \overline{\ran ((\mu  \tilde E-\tilde A)\inv \tilde E)^p} = \overline{\ran R_l(\mu)^p}.\qedhere
    \end{equation*}
\end{proof}

\begin{defi}[Integrated Semigroup]\hfill\\
    We call $(S_r(t))_{t\geq 0}$ and $(S_l(t))_{t\geq 0}$ from Theorem \ref{thm:1} and Corollary \ref{cor:1} the \textit{$p$-times integrated semigroups} of $(E,A)$ and call $(E,A)$ the \textit{generator} of $(S_r(t))_{t\geq 0}$ and $(S_l(t))_{t\geq 0}$.
\end{defi}

In the following we denote with $\omega$ the maximum of $\omega$ and $\tilde \omega$ from Theorem \ref{thm:1} and Corollary \ref{cor:1}.

\begin{lem}\label{lem:continuity-of-int-semigroup}\hfill\\
    Let $(E,A)$ satisfy Assumption \ref{assumption:1} and set $p:=p_{\mathrm{res}}^{(E,A)}+1$. For every $x_0\in X_{\ran}$ and $z_0\in Z_{\ran}$ the mappings $t\mapsto S_r(t)x_0$ and $t\mapsto S_l(t)z_0$ are continuous.
\end{lem}

\begin{proof}
    This is a direct consequence of the fact that the integrated semigroup is exponentially bounded and for elements in the dense subset $x_0\in \ran R_r(\mu)^p$ the map $t\mapsto S_r(t)x_0$ coincides with the $p$-th integral of the mild solution of \eqref{eq:dae} with initial value $x_0$.
\end{proof}

\begin{lem}\label{lem:properties-integrated-semigroup}\hfill\\
    Let $(E,A)$ satisfy Assumption \ref{assumption:1}, set $p:=p_{\mathrm{res}}^{(E,A)}+1$ and let $(S_l(t))_{t\geq 0}$ and $(S_r(t))_{t\geq 0}$  be the $p$-times integrated semigroups of $(E,A)$.
    \begin{enumerate}[label=({\alph*)}]
        \item \label{lem:properties-integrated-semigroup-a}For all $\mu \in \rho(E,A)\cap \C_{\Re >\omega}$ and $t\geq 0$ it holds that $R_r(\mu)S_r(t) = S_r(t)R_r(\mu)$ on $X_{\ran}$ and $R_l(\mu)S_l(t) = S_l(t)R_l(\mu)$ on $Z_{\ran}$.
        
        \item \label{lem:properties-integrated-semigroup-b}For all $t \geq 0$ one has
        \begin{equation*}
          ES_r(t)x_0 = S_l(t)Ex_0, \qquad x_0 \in X_{\ran},
        \end{equation*}
        and
        \begin{equation*}
          AS_r(t)z_0 = S_l(t)Az_0, \qquad z_0 \in \ran R_r(\mu)^p.
        \end{equation*}
        
        \item \label{lem:properties-integrated-semigroup-c}$(S_r(t))_{t\geq 0}$ maps all elements from $x_0 \in \ran R_r(\mu)^p$ to the solution of the \textit{$p$-times integrated differential-algebraic equation}, i.e.
        \begin{equation} \label{eq:integrated-cp}
            \frac{\dd}{\dd t}ES_r(t) x_0 = A S_r(t) x_0 + \frac{t^{p-1}}{(p-1)!}Ex_0.
        \end{equation}
        
        \item \label{lem:properties-integrated-semigroup-d}For all $x\in X_{\ran}$ one has $\int_0^t S_r(\tau)x\dx[\tau]\in \dom(A)$ and, in particular, 
        \begin{equation*}
            A\int_0^t S_r(\tau)x\dx[\tau] = ES_r(t)x - \frac{t^p}{p!}Ex, \quad x\in X_{\ran}.
        \end{equation*} 
        
        \item \label{lem:properties-integrated-semigroup-e}One has
        \begin{equation*}
            A(\mu E-A)\inv \int_0^t S_l(\tau)z\dx[\tau] = R_l(\mu)S_l(t)z - \frac{t^p}{p!}R_l(\mu)z, \quad z\in Z_{\ran}.
        \end{equation*}
        
        \item \label{lem:properties-integrated-semigroup-f}
        One has
        \begin{align*}
            S_r(t)S_r(s) &= \int_0^t \frac{1}{(p-1)!} \left( (t-\tau)^{p-1} S_r(\tau+s) - (t+s-\tau)^{p-1} S_r(\tau) \right) \dx[\tau] \\ 
            &= \int_0^t \frac{(t-\tau)^{p-1}}{(p-1)!} \left( S_r(\tau+s)-S_r(\tau)\right) \dx[\tau] - \sum_{k=1}^{p-1}\frac{s^k}{k!} \int_0^t \frac{(t-\tau)^{p-1-k}}{(p-1-k)!} S_r(\tau)\dx[\tau], \\
            S_l(t)S_l(s) &= \int_0^t \frac{1}{(p-1)!} \left( (t-\tau)^{p-1} S_l(\tau+s) - (t+s-\tau)^{p-1} S_l(\tau) \right) \dx[\tau] \\ 
            &= \int_0^t \frac{(t-\tau)^{p-1}}{(p-1)!} \left( S_l(\tau+s)-S_l(\tau)\right) \dx[\tau] - \sum_{k=1}^{p-1}\frac{s^k}{k!} \int_0^t \frac{(t-\tau)^{p-1-k}}{(p-1-k)!} S_l(\tau)\dx[\tau]
        \end{align*}
        on $X_{\ran}$ and $Z_{\ran}$, respectively. Moreover, for all $x_0 \in X_{\ran}$ and $z_0 \in Z_{\ran}$, it holds that
        \begin{equation*}
          S_r(t)S_r(s)x_0 = S_r(s)S_r(t)x_0,
          \qquad
          S_l(t)S_l(s)z_0 = S_l(s)S_l(t)z_0.
        \end{equation*}
    \end{enumerate}
\end{lem}

For $X=Z$ and $E=I_X$ the equation \eqref{eq:integrated-cp} is known as the $p$-times integrated Cauchy problem (see \cite{neubrander_integrated_1988, arendt_vector-valued_2011}).

\begin{proof}\
    \begin{enumerate}[label=({\alph*)}]
        \item Let $\mu, \lambda \in \rho(E,A)\cap \C_{\Re >\omega}$ and $x_0 \in X_{\ran}$. Since $R_r(\mu) R_r(\lambda) = R_r(\lambda) R_r(\mu)$ one has
        \begin{equation*}
            \mu^p \int_0^\infty \e^{-\mu t} R_r(\lambda) S(t) x_0 \dx[t] =R_r(\lambda)R_r(\mu)x_0 = R_r(\mu)R_r(\lambda)x_0 = \mu^p \int_0^\infty \e^{-\mu t} S(t) R_r(\lambda)x_0 \dx[t].
        \end{equation*}
        Then the uniqueness of the Laplace transform implies the assertion. It can be analogously obtained that $R_l(\mu)S_l(t) = S_l(t)R_l(\mu)$ on $Z_{\ran}$.
        \item Let $x_0\in X_{\ran}$. Then $Ex_0\in Z_{\ran}$, and, since $E$ is bounded, we have
        \begin{equation*}
            \mu^p \int_0^\infty \e^{\mu t} ES_r(t)x_0 \dx[t] = E R_r(\mu)x_0 = R_l(\mu) Ex_0 = \mu^p \int_0^\infty \e^{\mu t} S_l(t) Ex_0\dx[t],
        \end{equation*}
        and we conclude from uniqueness of the Laplace transform that $ES_r(t)x_0 = S_l(t) Ex_0$. Analogously, for $x_0 \in \ran R_r(\mu)^p$ one has $AS_r(t) x_0= S_l(t) Ax_0$. 
        
        \item This follows by exactly the same argumentation as in the proof of Theorem \ref{thm:1}.
        
        \item Let $\lambda, \mu \in \rho(E,A)\cap\C_{\Re >\omega}$ and $x_0\in \ran R_r(\lambda)^p$. Since $(\mu E-A)\inv$ is bounded \eqref{eq:integrated-cp} is equivalent to
        \begin{align*}
            \frac{\dd}{\dd t} R_r(\mu) S_r(t)x_0 &= (\mu E-A)\inv A S_r(t)x_0 + \frac{t^{p-1}}{(p-1)!} R_r(\mu)x_0 \\ 
            &= \mu R_r(\mu) S_r(t)x_0 - S_r(t)x_0 + \frac{t^{p-1}}{(p-1)!}R_r(\mu)x_0,
        \end{align*}
        where we have used that $(\mu E-A)\inv A = \mu R_r(\mu) - I_{\dom(A)}$. An integration gives, by using $S_r(0)=0$,
        \begin{align}\label{eq:integral-integrated-semigroup}
            \int_0^t S_r(\tau)x_0 \dx[\tau] = -R_r(\mu) S_r(t) x_0 + \mu R_r(\mu) \int_0^t S_r(\tau)x_0\dx[\tau] + \frac{t^p}{p!} R_r(\mu)x_0,
        \end{align}
        and especially $\int_0^t S_r(\tau)x_0\dx[\tau] \in \dom(A)$. Since $R_r(\mu)$ is bounded and $(S_r(t))_{t\geq 0}$ is exponentially bounded, \eqref{eq:integral-integrated-semigroup} extends to all $x\in \overline{R_r(\mu)^p} = X_{\ran}$.
        Note that the right-hand side of \eqref{eq:integral-integrated-semigroup} maps into $\dom(A)$.
        Multiplying both sides with $(\mu E-A)$ and rearranging the equation a last time, we obtain the assumption.
        
        \item For $\lambda \in \rho(E,A)\cap \C_{\Re >\omega}$ we define $\tilde E \coloneqq E (\lambda E-A)\inv$ and $\tilde A \coloneqq A(\lambda E-A)\inv$. The remainder follows analogously to the proof of \ref{lem:properties-integrated-semigroup-d}.
        
        \item The proof proceeds analogously to \cite[Prop.~5.1]{neubrander_integrated_1988} for $(S_r(t))_{t\geq 0}$. 
        For $(S_l(t))_{t\geq 0}$, one first transforms \eqref{eq:dae} into
        \begin{equation*}
          \tfrac{\dd}{\dd t}\, E(\mu E-A)^{-1} z(t) = A(\mu E-A)^{-1} z(t),
        \end{equation*}
        as in Corollary~\ref{cor:1}, and then repeats the previous steps. 
        Finally, the commutativity of the integrated semigroups can be verified directly by straightforward calculations.        \qedhere
    \end{enumerate}
\end{proof}

\section{Solvability of inhomogeneous DAEs}\label{section:inhomogeneous-daes}
Next we consider the more general case of DAEs with inhomogeneity, i.e.,
\begin{equation}\label{eq:inh-without-initial-value}    \begin{split}
        \tfrac{\dd}{\dd t}Ex(t) &= Ax(t) + f(t), \quad t\in [0,\infty),
    \end{split}
\end{equation}
where $f\colon [0,\infty)\to Z$ and $X$, $Z$, $E$ and $A$ as in Assumption \ref{assumption:1}. 
As we are only examining linear systems, we are going to neglect the initial value for now and only give a presentation for the particular solution of \eqref{eq:inh-without-initial-value}. Before analyzing the solutions for any arbitrary function $f$, we first distinguish between two scenarios. Specifically, we examine the cases where $f(t)$ resides either in $\overline{\ran R_l(\mu)^p}$ or in $\ker R_l(\mu)^p$. The latter case is derived from the finite dimensional setting \cite[Ch.~2.3]{kunkel_differential-algebraic_2006}, with a few additional steps. 

\subsection{Inhomogeneity in \texorpdfstring{$Z_{\ker}$}{Zker}}
We start by showing that $A$ restricted to $Z_{\ker} = \ker R_r(\mu)^p$ is invertible, as long as the complex resolvent index exists. This result can be found in \cite[Ch.~2.2]{sviridyuk_linear_2003} under stronger conditions, namely by assuming weak radiality of $(E,A)$ \cite[Ch.~2.1]{sviridyuk_linear_2003}. To simplify the notation we set $X_{\ker} \coloneqq \ker R_r(\mu)^p$ and $Z_{\ker} \coloneqq \ker R_l(\mu)^p$.

The following proof makes use of the index theory for infinite-dimensional systems. In particular, it uses the chain index (see \cite{berger_quasi-weierstrass_2012} for the finite dimensional case or \cite{erbay_index_2024, sviridyuk_linear_2003} for the infinite-dimensional case). To briefly recall the definition, we say $x_1, \ldots, x_q\in X$ is a \textit{chain of $(E,A)$ of length $q$}, if $x_1\in \ker E\backslash \{0\}$ and $E x_{i+1} = Ax_i$, $i=1,\ldots,q-1$. Further, we denote the supremum over all chain lengths of $(E,A)$ as the \textit{chain index of $(E,A)$}.

\begin{lem}\label{lem:A-invertible-on-ker}\hfill\\
    Let $(E,A)$ satisfy Assumption \ref{assumption:1} and set $p:=p_{\mathrm{res}}^{(E,A)}+1$.
    Then, the operator 
    \begin{equation*}
        A_{\ker}\colon \dom(A_{\ker})\subseteq X_{\ker} \to Z_{\ker}, \quad x\mapsto Ax,
    \end{equation*}
    with
    \begin{equation*}
        \dom(A_{\ker})\coloneqq \dom(A)\cap X_{\ker}
    \end{equation*}
    has a bounded inverse. Furthermore, $E_{\ker} A\inv_{\ker} \in L(Z_{\ker})$ and $A\inv_{\ker} E_{\ker} \in L(X_{\ker})$ are nilpotent of degree not exceeding $p$, where $E_{\ker}\coloneqq E|_{X_{\ker}} \in L(X_{\ker}, Z_{\ker})$.
\end{lem}

\begin{proof}
    Let $x\in \dom(A_{\ker})$. By \cite[Remark 2.1.4]{sviridyuk_linear_2003} we have
    \begin{equation*}
        R_l(\mu)^p(A_{\ker} x) = A (R_r(\mu)^p x) = 0.
    \end{equation*}
    Thus, $A_{\ker}$ is well-defined. Analogously, by using that $E$ is bounded, 
    \begin{equation*}
        R_l(\mu)^p(Ex) = E ( R_r(\mu)^p x)\;\text{ for all $x\in X_{\ker}$,}
    \end{equation*}
    we obtain that $E_{\ker} \in L(X_{\ker}, Z_{\ker})$. 
	
    Next, we show that the chain index of $(E,A)$ is bounded by $p$. Let $(x_1, \ldots, x_q)$ be a chain of $(E,A)$ of length $q=p+1 = p_{\mathrm{res}}^{(E,A)}+2$. By \cite[Thm.~2.1.3]{sviridyuk_linear_2003} one has
    \begin{equation*}
        - R_r(\mu) x_{q} =  x_{q-1} + \mu x_{q-2} + \ldots  + \mu^{q-2}x_1,
    \end{equation*}
    for some $\mu \in \rho(E,A)$. Rearranging this equation and using the complex resolvent index one obtains
    \begin{equation*}
        \Vert x_1 \Vert \leq \frac{C \vert \mu \vert^{p_{\mathrm{res}}^{(E,A)} -1}}{\vert \mu \vert^{q-2}} \Vert x_q\Vert + \frac{1}{\vert\mu\vert^{q-2}} \Vert x_{q-1} \Vert + \ldots + \frac{1}{\vert \mu \vert} \Vert x_2\Vert.
    \end{equation*}
    Assuming $\mu \in \R$ and letting $\mu \to \infty$, one obtains $x_1 = 0$. This is a contradiction and shows that all chains of $(E,A)$ are at most of length $p$, i.e.~$p_{\mathrm{chain}}^{(E,A)}\leq p$.
	
    By \cite[Thm.~2.1.2]{sviridyuk_linear_2003} $X_{\ker}$ is the linear span of all chains of $(E,A)$ whose lengths do not exceed $p$. Then the rest follows analogously to the proofs of \cite[Lem.~2.2.4\&2.2.5]{sviridyuk_linear_2003}.
\end{proof}

\begin{remark}
    In the proof of Lemma \ref{lem:A-invertible-on-ker} one only uses the resolvent growth of $(E,A)$ along the real-axis. Therefore, it would be sufficient to assume only 
    \begin{equation*}
        \Vert (\lambda E-A)\inv \Vert \leq C \vert \lambda \vert^p, \quad \lambda \in [\omega,\infty),
    \end{equation*}
    for a $\omega \in \R$, $C>0$. This is also known as the \textit{(real) resolvent index} (see \cite{erbay_index_2024}). Further, it becomes clear that the chain index is always bounded by the (real) resolvent index $+1$ (see \cite{erbay_index_2024} for further information regarding the index of a DAE).
\end{remark}

With Lemma \ref{lem:A-invertible-on-ker} we now have all the tools to present solutions with inhomogeneities in $Z_{\ker}$ (see \cite[Ch.~2]{kunkel_differential-algebraic_2006} for finite dimensions).

\begin{prop}\label{prop:sol-with-inh-in-ker}\hfill\\
    Let $(E,A)$ satisfy Assumption \ref{assumption:1} and set $p:=p_{\mathrm{res}}^{(E,A)}+1$. Let $f\in C^p([0,\infty); Z_{\ker})$. Then \eqref{eq:inh-without-initial-value} has a classical solution of the form
    \begin{equation}\label{eq:sol-with-inh-in-ker}
        x(t) = - \sum_{i=0}^{p} A\inv_{\ker} (E_{\ker} A\inv_{\ker})^i f^{(i)}(t), \quad t\geq 0.
    \end{equation}
\end{prop}

\begin{proof}
    By applying $E$ to the solution $x(t)$ one obtains
    \begin{equation*}
        Ex(t) = - \sum_{i=0}^p (E_{\ker} A\inv_{\ker})^{i+1} f^{(i)}(t) = - \sum_{i=0}^{p-1} (E_{\ker} A\inv_{\ker})^{i+1} f^{(i)}(t), \quad t\geq 0.
    \end{equation*}
    Note that the nilpotency of $E_{\ker} A\inv_{\ker}$ was used in the second equation. Since $f$ is $p$-times differentiable, $Ex$ becomes differentiable as well and one has
    \begin{align*}
        \frac{\dd}{\dd t}Ex(t) &= - \sum_{i=0}^{p-1} (E_{\ker} A\inv_{\ker})^{i+1}f^{(i+1)}(t) = - \sum_{i=1}^{p} (E_{\ker} A\inv_{\ker})^if^{(i)}(t)  \\
        &= - A_{\ker} \left( \sum_{i=1}^p A\inv_{\ker} (E_{\ker} A\inv_{\ker})^i f^{(i)}(t) \right) \pm f(t) = A x(t) + f(t). \qedhere
    \end{align*}
\end{proof}

\begin{remark}\label{remark:sol-with-inh-in-ker}
By \eqref{eq:sol-with-inh-in-ker}, it is clear that the initial value of the classical solution depends on the derivatives of $f$ under the operators $A^{-1}_{\ker} (E_{\ker} A^{-1}_{\ker})^i$.
In particular, if $f \in H^p_{0,l}([0,\infty); Z_{\ker})$ (as defined in \eqref{eq:H0}), then the initial value of the solution necessarily vanishes at  $t=0$.
\end{remark}

\subsection{Inhomogeneity in \texorpdfstring{$Z_{\ran}$}{Zran}}
In order to deal with inhomogeneities in $Z_{\ran}$, we are going to generalise the solution theory of the inhomogeneous abstract Cauchy problem $\frac{\dd}{\dd t} x(t) = Ax(t) + f(t)$, $x(0) = x_0$, and extend the proofs of \cite[Chapter~2.3]{arendt_vector-valued_2011} to the inhomogeneous DAE case with the help of Lemma \ref{lem:properties-integrated-semigroup}.

\begin{lem}\label{lem:solutions-live-in-xran}\hfill\\
    Let $(E,A)$ satisfy Assumption \ref{assumption:1} and set $p:=p_{\mathrm{res}}^{(E,A)}+1$. Let $f\in L^1([0,\infty);Z_{\ran})$ and let $x\colon [0,\infty)\to X$ be a continuous mild solution of \eqref{eq:dae-0}. Then $x(t) \in X_{\ran}$ for all $t\geq 0$.
\end{lem}

\begin{proof}
    For simplicity, assume that $\lambda = 0\in \rho(E,A)$.
    Let $x$ be a mild solution of \eqref{eq:dae-0}. Then, for $t,h\geq 0$,
    \begin{equation*}
        Ex(t+h)-Ex(t) = A\int_t^{t+h} x(\tau)\dx[\tau] + \int_t^{t+h}f(\tau)\dx[\tau],
    \end{equation*}
    or equivalently
    \begin{equation}\label{eq:4}
        A\inv E\left(x(t+h)-x(t)\right)- A\inv \int_t^{t+h} f(\tau)\dx[\tau] = \int_t^{t+h}x(\tau)\dx[\tau].
    \end{equation}
    Observe that
    \begin{equation}\label{eq:5}
        A\inv (Z_{\ran}) \subseteq X_{\ran}\subseteq \overline{\ran R_r(\lambda)^k}, \quad k \in \N_0.
    \end{equation}
    Thus, $A\inv \int_t^{t+h}f(\tau)\dx[\tau] \in X_{\ran}$. Altogether, the left-hand side of \eqref{eq:4} maps into $\overline{\ran R_r(\lambda)}$. By dividing both sides by $h$ and letting $h\to 0$, one deduces $x(t)\in \overline{\ran R_r(\lambda)}$.
    Repeating this iteratively, one obtains $x(t)\in \overline{\ran R_r(\lambda)^k}$, for $k\in \N$. Thus,
    \begin{equation*}
        x(t)\in \bigcap_{k\in \N_0} \overline{\ran R_r(\lambda)^k}.
    \end{equation*}
    As seen in \cite[Prop.~5.1]{sviridyuk_linear_2003} the sequence $\overline{\ran R_r(\lambda)^k}$ stagnates at $k=p$. This leads to $x(t) \in \overline{\ran R_r(\lambda)^p} = X_{\ran}$.
\end{proof}

Since this is the case we are interested in, we assume $f\colon[0,\infty) \to Z_{\ran}$. Further, we define for $x_0 \in X_{\ran}$ the function
\begin{equation}\label{eq:inhom-int-sol}
    v(t)\coloneqq S_l(t)Ex_0 + \int_0^t S_l(s)f(t-\tau)\dx[\tau], \quad t\in [0, \infty),
\end{equation}
where $(S_l(t))_{t\geq 0}$ denotes the $p$-times integrated semigroup from Corollary \ref{cor:1}

\begin{lem}\label{lem:solution-equals-derivative-of-v}\hfill\\ 
    Let $(E,A)$ satisfy Assumption \ref{assumption:1}, set $p:=p_{\mathrm{res}}^{(E,A)}+1$ and let $(S_l(t))_{t\geq 0}$ and $(S_r(t))_{t\geq 0}$  be the $p$-times integrated semigroups of $(E,A)$. Let $f\in L^1([0, \infty); Z_{\ran})$, $x_0 \in \ran R_r(\mu)^p$ and $v$ as in \eqref{eq:inhom-int-sol}.
    \begin{enumerate}[label=({\alph*)}]
        \item If there exists a mild solution $x\colon [0,\infty) \to X$ of \eqref{eq:dae-0}, then $v\in C^p([0,\infty); Z_{\ran})$ and $Ex = v^p$. \label{lem:solution-equals-derivative-of-v-a}
        \item If there exists a classical solution $x\colon [0,\infty)\to X$ of \eqref{eq:dae-0}, then $v \in C^{p+1}([0,\infty) ; Z_{\ran})$ and $\frac{\dd}{\dd t}Ex = v^{p+1}$.
    \end{enumerate}
\end{lem}

\begin{proof}
    \begin{enumerate}[label=({\alph*)}]
        \item Let $0\leq \tau \leq t <\infty$ and define $w(\tau)\coloneqq ES_r(t-\tau) \int_0^\tau x(\sigma)\dx[\sigma]$. By Lemma \ref{lem:solutions-live-in-xran} this is indeed well defined, since $x$ is a continuous mild solution and $x(t) \in X_{\ran}$ for all $t\geq 0$. In addition $S_r(t)$ is closed (as a bounded operator) and, thus, $\int_0^\tau x(\sigma)\dx[\sigma] \in \ran R_r(\mu)^p\cap\dom(A)$ for all $\tau\geq 0$. Hence, by Lemma \ref{lem:properties-integrated-semigroup} \ref{lem:properties-integrated-semigroup-b} and \ref{lem:properties-integrated-semigroup-c} we have
        \begin{align*}
            \frac{\dd}{\dd t}w(\tau) &= -AS_r(t-\tau)\int_0^\tau x(\sigma)\dx[\sigma] - \frac{(t-\tau)^{p-1}}{(p-1)!}E\int_0^\tau x(\sigma)\dx[\sigma] + ES_r(t-\tau)x(\tau)\\
            &=S_l(t-\tau)\left( Ex(\tau)-A\int_0^\tau x(\sigma)\dx[\sigma] \right) - \frac{(t-\tau)^{p-1}}{(p-1)!}E\int_0^\tau x(\sigma)\dx[\sigma]\\
            &= S_l(t-\tau) \left( \int_0^\tau f(\sigma)\dx[\sigma] + Ex(0)\right) - \frac{(t-\tau)^{p-1}}{(p-1)!}E\int_0^\tau x(\sigma)\dx[\sigma].
        \end{align*}
        Since $\int_0^t \tfrac{\dd}{\dd t}w(\tau)\dx[\tau] = w(t)-w(0)=0$ we obtain
        \begin{equation*}
            \int_0^t S_l(t-\tau) \left( \int_0^\tau f(\sigma)\dx[\sigma] + Ex(0)\right) \dx[\tau] = \int_0^t \frac{(t-\tau)^{p-1}}{(p-1)!} E\int_0^\tau x(\sigma)\dx[\sigma] \dx[\tau].
        \end{equation*}
        Thus, by taking the $(p+1)$-th derivative, we obtain
        \begin{align*}
            Ex(t) &= \frac{\dd^{p+1}}{\dd t^{p+1}} \left( \int_0^t \frac{(t-\tau)^{p-1}}{(p-1)!} E\int_0^\tau x(\sigma)\dx[\sigma] \dx[\tau] \right)\\
            &= \frac{\dd^{p+1}}{\dd t^{p+1}} \left( \int_0^t S_l(t-\tau) \left( \int_0^\tau f(\sigma)\dx[\sigma] + Ex(0)\right) \dx[\tau]\right)\\
            &= \frac{\dd^{p+1}}{\dd t^{p+1}} \left( \int_0^t S_l(\tau) \left( \int_0^{t-\tau} f(\sigma)\dx[\sigma] + Ex(0)\right) \dx[\tau]\right)\\
            &= \frac{\dd^p}{\dd t^p} \left( \int_0^t S_l(\tau) f(t-\tau)\dx[\tau] + S_l(t)Ex(0)\right) = \frac{\dd^p}{\dd t^p}v(t).
        \end{align*}
        Note that we used \cite[Prop.~1.3.6]{arendt_vector-valued_2011} in the last equation.
        
        \item By \cite[Thm.~2.1]{gernandt_pseudo-resolvent_2023} every classical solution is also a mild solution and $Ex$ is continuously differentiable. Together with Part \ref{lem:solution-equals-derivative-of-v-a} we obtain $Ex(t)=\frac{\dd^p}{\dd t^p}v(t)$ and thus $v\in C^{p+1}([0,\infty); Z_{\ran})$. \qedhere
    \end{enumerate}
\end{proof}

Next, we generalise \cite[Lem.~3.2.10]{arendt_vector-valued_2011}. To do so, we have to assume, that $f$ maps into the range of $E$. Thus, we assume that there exists a $\tilde f\in L^1([0,\infty); X_{\ran})$, such that $f= E \tilde f$. Furthermore, we define
\begin{equation}\label{eq:inhom-int-sol-2}
    \tilde v(t) \coloneqq S_r(t) x_0 + \int_0^t S_r(t-\tau) \tilde f(\tau) \dx[\tau], \quad  t \geq 0.
\end{equation}
By Lemma \ref{lem:properties-integrated-semigroup} \ref{lem:properties-integrated-semigroup-b} we have $ES_r(t) = S_l(t) E$ and, thus, $v(t) = E \tilde v(t)$. 

\begin{lem}\label{lem:solution-equals-derivative-of-v-part-2}\hfill\\ 
    Let $(E,A)$ satisfy Assumption \ref{assumption:1}, set $p:=p_{\mathrm{res}}^{(E,A)}+1$ and let $(S_l(t))_{t\geq 0}$ and $(S_r(t))_{t\geq 0}$  be the $p$-times integrated semigroups of $(E,A)$.
    Let $\tilde f\in L^1([0,\infty); X_{\ran})$ and $f=E\tilde f$, $x_0 \in \ran R_r(\mu)^p$ and $\tilde v$ as in \eqref{eq:inhom-int-sol-2}. 
    \begin{enumerate}[label=({\alph*)}]
        \item If $\tilde v \in C^p([0,\infty);X_{\ran})$, then $x=\frac{\dd^p}{\dd t^p}\tilde v$ is a mild solution  of \eqref{eq:dae-0}.
        \item If $\tilde v \in C^{p+1}([0,\infty);X_{\ran})$, then $x=\frac{\dd^p}{\dd t^p}\tilde v$ is a classical solution  of \eqref{eq:dae-0}.
    \end{enumerate}
\end{lem}

\begin{proof}
    \begin{enumerate}[label=({\alph*)}]
        \item We start by integrating $\tilde v$ and using Fubini's Theorem
        \begin{align*}
            \int_0^t \tilde v(\tau)\dx[\tau] &= \int_0^t S_r(\tau) x_0 \dx[\tau] + \int_0^t \int_0^\tau S_r(\sigma) \tilde f(\tau-\sigma)\dx[\sigma]\dx[\tau]\\ 
            &= \int_0^t S_r(\tau) x_0 \dx[\tau] + \int_0^t \int_0^{t-\sigma} S_r(\tau) \tilde f(\sigma) \dx[\tau] \dx[\sigma].
        \end{align*}
        Applying $A$ from the left and using Lemma \ref{lem:properties-integrated-semigroup} \ref{lem:properties-integrated-semigroup-d} we deduce
        \begin{align}\label{eq:1}
            \nonumber A \int_0^t \tilde v(\tau)\dx[\tau] &= E S_r(t) x_0 - \frac{t^p}{p!} Ex_0 + \int_0^t E S_r(t-\tau) \tilde f(\tau) - \frac{(t-\tau)^p}{p!} \underbrace{E \tilde f(\tau)}_{=f(\tau)} \dx[\tau] \\
            &= E \tilde v(t) - \frac{t^p}{p!} Ex_0 - \int_0^t \frac{(t-\tau)^p}{p!} f(\tau)\dx[\tau].
        \end{align}
        Since $E\tilde v = v \in C^p([0,\infty); Z_{\ran})$, $\tilde v(0)=0$ and $A$ is closed, we deduce
        \begin{equation}\label{eq:2}
            A \tilde v^{(k)}(t) = E \tilde v^{(k+1)}(t) - \frac{t^{p-(k+1)}}{(p-(k+1))!} Ex_0 - \int_0^t \frac{(t-\sigma)^{p-(k+1)}}{(p-(k+1))!} f(\sigma) \dx[\sigma],
        \end{equation}
        for $k\in \{ 0,\ldots, p-1\}$. Hence, for $k=p-1$ we have
        \begin{equation*}
            A \tilde v^{(p-1)}(t) = E \tilde v^{(p)}(t) - Ex_0 - \int_0^t f(\sigma)\dx[\sigma].
        \end{equation*}
        Thus, $x(t)\coloneqq \tilde v^{(p)}(t)$ is a mild solution of \eqref{eq:dae-0}.
        
        \item If $E\tilde v = v\in C^{p+1}([0,\infty); Z_{\ran})$, we derive
        \begin{equation*}
            A\tilde v^{(p)}(t) = E\tilde v^{(p+1)}(t) - f(t).
        \end{equation*}
        Together with \eqref{eq:1} and \eqref{eq:2} we conclude $E\tilde v(0) = \ldots = E \tilde v^{(p-1)} (0)=0$ and $E\tilde v^{(p)}(0)= Ex_0$. Thus, $x(t) = \tilde v^{(p)}(0)$ is a classical solution of \eqref{eq:dae-0}. \qedhere
    \end{enumerate}
\end{proof}

The condition $f=E\tilde f$ may appear somewhat restrictive at first.
However, one should keep in mind that we are already restricting the system to $X_{\ran}$ and $Z_{\ran}$.
Under the additional assumptions that $X=Z$ and that $\ran E$ is closed, we always have $Z_{\ran}\subseteq \ran E$.
Consequently, every $f \in L^1([0,\infty); Z_{\ran})$ lies in the range of $E$.

In what follows, we specify a condition ensuring the existence of such a function $\tilde f$.
\begin{prop}\label{prop:E-on-X_ran-is-bd-inv}\hfill\\
    Let $(E,A)$ satisfy Assumption \ref{assumption:1} and set $p:=p_{\mathrm{res}}^{(E,A)}+1$.
    If $E(X_{\ran})$ is closed and $X_{\ran} \cap \ker R_r(\mu)^p = \{0\}$, then $E\colon X_{\ran} \to Z_{\ran}$ is boundedly invertible. 
\end{prop}

\begin{proof}
     Since $E(\ran R_r(\mu)^p)\subseteq \ran R_l(\mu)^p$, it is easy to see that $E(X_{\ran})\subseteq Z_{\ran}$. Now, let $z\in \ran R_l(\mu)^{p+1}$. Then there exists some $y\in Z$ with
     \begin{equation*}
         z=R_l(\mu)^{p+1}y = ER_r(\mu)^p (\mu E-A)\inv y.
     \end{equation*}
     Hence,
     \begin{equation*}
         \ran R_l(\mu)^{p+1}\subseteq E(\ran R_r(\mu)^p)\subseteq E(X_{\ran}),
     \end{equation*}
     and a~closure gives $Z_{\ran}\subseteq \overline{E(X_{\ran})} = E(X_{\ran})$. Therefore, $E\colon X_{\ran} \to Z_{\ran}$ is surjective. 
     Since $X_{\ran} \cap \ker R_r(\mu)^p = \{0\}$ and $\ker E\subseteq \ker R_r(\mu)^p$, $E\colon X_{\ran} \to Z_{\ran}$ is injective. Thus, the statement follows from the open mapping theorem.
\end{proof}


\subsection{An Example with \texorpdfstring{$X\neq X_{\ran} \oplus X_{\ker}$}{}}\label{section:example}

In the following example, we highlight that the spaces $X$ and $Z$ do not always decompose into $X_{\ran}\oplus X_{\ker}$ and $Z_{\ran}\oplus Z_{\ker}$, respectively. Thus, it is not always possible to represent the solution of the inhomogeneous system \eqref{eq:inh-without-initial-value} using Proposition \ref{prop:sol-with-inh-in-ker} and Lemma \ref{lem:solution-equals-derivative-of-v}.


We consider the following system
    \begin{align*}
        & & \tfrac{\partial}{\partial t}x_1(\xi, t) &= -\tfrac{\partial}{\partial \xi} x_1(\xi, t),  &\xi \in (0,1), t\geq 0,& & \\
        & & 0 &= -\tfrac{\partial}{\partial \xi}x_2(\xi,t),  &\xi \in (1,2), t\geq 0,& & \\ 
        & & x_1(0,t) &= 0,  & t\geq 0,& & \\
        & & x_1(1,t) &= x_2(1,t),  & t\geq 0.& &
    \end{align*}
    Next we rewrite this as an infinite-dimensional DAE. To this end, we introduce the evaluation operator at $\xi\in[a,b]$ by $\delta_\xi\in H^1(a,b)^\ast$.
 Then the above system is represented by a~DAE with operators $E\colon X\to Z$ and $A\colon \dom(A)\subseteq X\to Z$ with
    \begin{equation*}
        E=\begin{bmatrix}
            I & 0\\ 0 & 0\\ 0 & 0 \\ 0 & 0
        \end{bmatrix}, \qquad A=\begin{bmatrix}
            -\frac{\dd}{\dd \xi} & 0\\ 0 & -\frac{\dd}{\dd \xi}\\ -\delta_0 & 0\\ -\delta_1 & \delta_1
        \end{bmatrix},
    \end{equation*}
    for
    \begin{align*}
        \dom(A) &= H^1(0,1)\times H^1(1,2),\\
        X &= L^2(0,1)\times L^2(1,2), & Z&=L^2(0,1)\times L^2(1,2)\times \C^2.
    \end{align*}
    For $\lambda \in \C_{\Re \geq 0}$ we compute
    \begin{equation}\label{ex:equation1}
        (\lambda E-A)\inv \begin{bmatrix}
            f \\g \\ \mu_1 \\ \mu_2
        \end{bmatrix} = \begin{bmatrix}
            \displaystyle \mu_1 \e^{-\lambda \cdot} + \int_0^\cdot \e^{-\lambda(\cdot-\sigma)} f(\sigma)\dx[\sigma]\\
            \displaystyle -\mu_2 + \mu_1 \e^{-\lambda} + \int_0^1 \e^{-\lambda(1-\sigma)} f(\sigma)\dx[\sigma] + \int_1^\cdot g(\sigma)\dx[\sigma]
        \end{bmatrix}, \quad \begin{bmatrix}
            f \\g \\ \mu_1 \\ \mu_2
        \end{bmatrix} \in Z.
    \end{equation}
    Since the right-hand side of \eqref{ex:equation1} is bounded with respect to $\lambda$, the complex resolvent index $p_{\mathrm{res}}^{(E,A)}$ is at most $1$. 
    
    Next, we determine the nullspaces and the closure of the ranges of $(\lambda E-A)\inv E$ and $E(\lambda E-A)\inv$ in order to show that $X$ and $Z$ do not split into $X_{\ran}$, $X_{\ker}$ and $Z_{\ran}$, $Z_{\ker}$, respectively. 
    By \eqref{ex:equation1} we have
    \begin{equation}\label{ex:right-resolvent}
        (\lambda E-A)\inv E \begin{bmatrix}
            x_1 \\ x_2
        \end{bmatrix} = \begin{bmatrix}
            \displaystyle \int_0^\cdot \e^{-\lambda(\cdot-\sigma)}x_1(\sigma)\dx[\sigma]\\
            \displaystyle \int_0^1 \e^{-\lambda(1-\sigma)}x_1(\sigma)\dx[\sigma]
        \end{bmatrix}, \qquad \begin{bmatrix}
            x_1 \\x_2
        \end{bmatrix} \in X.
    \end{equation}
    Assume that $\int_0^\cdot \e^{-\lambda(\cdot-\sigma)}x_1(\sigma)\dx[\sigma] = 0$ in $L^2(0,1)$. By defining $x_1(t)=0$, $t\ge 1$, we obtain $x_1\in L^2(0,\infty)$ and $\e^{-\lambda \cdot}\ast x_1 =0$ in $ L^2(0,\infty)$. Applying the Laplace transform $\mathcal{L}$ to this equation yields  
    \begin{equation*}
        0 = \mathcal{L}(\e^{-\lambda \cdot} \ast x_1)(s) = \mathcal{L}(\e^{-\lambda \cdot})(s) \mathcal{L}(x_1)(s), \quad s\in \C_{\Re >0}.
    \end{equation*}
    Since $\mathcal{L}(\e^{-\lambda \cdot})(s)\neq 0$ for a suitable $\lambda$ and for all $s\in \C_{\Re >0}$ we deduce $\mathcal{L}(x_1)=0$ and, by the uniqueness of the Laplace transform, $x_1= 0$. Hence
    \begin{equation*}
        X_{\ker} = \ker (\lambda E-A)\inv E = \{0\} \times L^2(1,2).
    \end{equation*}
    
    Now, for simplicity we assume that $\lambda = 0$. This is possible as the kernels and ranges of the left- and right-resolvents are independent of the choice of $\lambda \in \rho(E,A)$ (see \cite[Lem.~2.1.2]{sviridyuk_linear_2003}). 
    Since 
    $$
    H^1_0(0,1) = \{ x \in L^2(0,1) \, \vert \, \exists x_1 \in L^2(0,1): x = \int_0^\cdot x_1(\sigma)\dx[\sigma] \; \text{and} \; x(0)=x(1)=0\},
    $$
    we observe $H^1_0(0,1)\times \{ 0 \} \subseteq \ran (\lambda E-A)\inv E$ and, since $H^1_0(0,1)$ is dense in $L^2(0,1)$,
    \begin{equation*}
        L^2(0,1)\times \{0\}\subseteq \overline{\ran (\lambda E-A)\inv E} = X_{\ran}.
    \end{equation*}
    Thus $X= X_{\ran} + X_{\ker}$. However, as the second entry of \eqref{ex:right-resolvent} is not empty, we see that
    $X_{\ran} \cap X_{\ker} \neq \{0\}$ meaning $X$ does not decompose into a direct sum of $X_{\ran}$ and $X_{\ker}$.
    
    Next, using \eqref{ex:equation1} we compute
    \begin{equation*}
        E(\lambda E-A)\inv \begin{bmatrix}
            f \\g \\ \mu_1 \\ \mu_2
        \end{bmatrix} = \begin{bmatrix}
            \displaystyle \mu_1 \e^{-\lambda \cdot} + \int_0^\cdot \e^{-\lambda(\cdot-r)} f(\sigma)\dx[\sigma]\\
            0 \\0 \\ 0
        \end{bmatrix}, \qquad \begin{bmatrix}
            f \\g \\ \mu_1 \\ \mu_2
        \end{bmatrix} \in Z.
    \end{equation*}
    To compute the kernel we assume $\mu_1 \e^{-\lambda \cdot} + \int_0^\cdot \e^{-\lambda(\cdot-\sigma)} f(\sigma)\dx[\sigma] = 0$ in $L^2(0,1)$. Inserting $t=0$ we deduce $\mu_1 = 0$ and $\int_0^\cdot \e^{-\lambda(\cdot-\sigma) f(\sigma)\dx[\sigma]} = \e^{-\lambda \cdot} \ast f = 0$. As already shown while computing the kernel of the right-resolvent, we deduce $f=0$. Thus, 
    \begin{equation*}
        \ker E (\lambda E-A)\inv = \{0\} \times L^2(1,2) \times \{0\} \times \C.
    \end{equation*}
    Further, similar to the steps before, one has
    \begin{equation*}
        Z_{\ran} = L^2(0,1)\times \{0\} \times \{0\} \times \{0\}.
    \end{equation*}
    Hence, in contrast to the space $X$, the spaces $Z_{\ran}$ and $Z_{\ker}$ are disjoint, but $Z_{\ran} \oplus Z_{\ker}\subsetneqq Z$.
    Note that it is not possible to apply Proposition \ref{prop:E-on-X_ran-is-bd-inv}, as $X_{\ran} \cap X_{\ker}$ is not empty. Since $E(X_{\ran}) = L^2(0,1)\times \{0\}^3$ is closed, it is still possible to find for every inhomogeneity $f$ on $Z_{\ran}$ a function $\tilde f\in L^2(0,1)\times \{0\}$, which is a subset of $X_{\ran}$, with $E\tilde f = f$. 

    Note that by choosing 
    \begin{equation*}
        E=\begin{bmatrix}
            I & 0\\ 0 & 0 \\ 0 & 0
        \end{bmatrix}, \qquad A=\begin{bmatrix}
            -\frac{\dd}{\dd \xi} & 0\\ 0 & -\frac{\dd}{\dd \xi}\\ -\delta_1 & \delta_1
        \end{bmatrix},
    \end{equation*}
    and $\dom(A)=\{(x_1, x_2)\in H^1(0,1)\times H^1(1,2) \, \vert \, x_1(0)=0\}$ one obtains $Z=Z_{\ran}\oplus Z_{\ker}$. 

\subsection{Hilbert space decomposition}\label{Section:hilbert-space-decomp}
In contrast to the preceding sections, we now assume that $X$ and $Z$ are complex Hilbert spaces.
It was shown in \cite{gernandt_pseudo-resolvent_2023} that these spaces can always be decomposed in such a way that the solutions of \eqref{eq:inh-without-initial-value} can be described explicitly for inhomogeneities lying outside $Z_{\ran}$.
This is possible whenever the spaces $\overline{\ran R_r(\mu)^k}$ and $\overline{\ran R_l(\mu)^k}$ stabilize for some $k \in \N_0$, which is guaranteed under the assumption of the existence of the complex resolvent index \cite[Prop.~5.1]{trostorff_semigroups_2020}.

To construct such a decomposition, we define
\begin{equation*}
    X_k\coloneqq \overline{\ran R_r(\mu)^k},\quad Z_k\coloneqq \overline{\ran R_l(\mu)^k}.
\end{equation*}
Starting with 
\begin{align*}
    X &= X_1 \oplus \ker (R_r(\mu))^\ast \eqqcolon X_1 \oplus W_{X,1},\\
    Z &= Z_1 \oplus \ker (R_l(\mu))^\ast \eqqcolon Z_1 \oplus W_{Z,1},
\end{align*}
we decompose the spaces $X_k$ and $Z_k$ by
\begin{align*}
    X_k &= X_{k+1} \oplus W_{X,k+1},\\ 
    Z_k &= Z_{k+1} \oplus W_{Z,k+1},\qquad k\in \{1,\ldots, p-2\},
\end{align*}
until the spaces $X_k$ and $Z_k$ stagnate. Here, $W_{X,k+1}$ and $W_{Z,k+1}$ denote the orthogonal complement of $X_{k+1}$ and $Z_{k+1}$ with respect to $X_k$ and $Z_k$, respectively. Thus, 
\begin{align}\label{eq:hilbert-space-decomposition}
    \begin{split}
        X &= X_{\ran} \oplus W_{X,p-1} \oplus W_{X,p-2} \oplus \ldots \oplus W_{X,1},\\
        Z &= Z_{\ran} \oplus W_{Z,p-1} \oplus W_{Z,p-2} \oplus \ldots \oplus W_{Z,1}.
    \end{split}
\end{align}
Let $\mu \in \rho(E,A)$. In \cite[Lem.~8.1]{gernandt_pseudo-resolvent_2023} it was shown that $x$ is a classical solution of \eqref{eq:inh-without-initial-value} with initial value $x_0$ if and only if $w_\mu=(A-\mu E)\e^{-\mu \cdot}x$ is a classical solution of 
\begin{align}\label{eq:left-resolvent-dae}
    \begin{cases}
        \frac{\dd}{\dd t} R_l(\mu) w_\mu(t) &= w_\mu(t) + \e^{-\mu t} f(t),\\
        R_l(\mu)w_\mu(0) &= R_l(\mu) (A-\mu E)x_0.
    \end{cases}
\end{align}
Additionally, in \cite[Sec.~5]{gernandt_pseudo-resolvent_2023} it was shown that $R_l(\mu)$ can be rewritten as
\begin{equation}\label{eq:left-resolvent-decomposition}
    R_l(\mu) = \begin{bmatrix}
        R_l(\mu)|_{Z_{\ran}} & R_l(\mu)|_{W_{Z, p-1}} & P_{Z_{\ran}} R_l(\mu)|_{W_{Z,p-2}} & \cdots & P_{X_{\ran}} R_l(\mu)|_{W_{Z,1}}\\
        0 & 0 & P_{W_{Z,p-1}} R_l(\mu)|_{W_{Z,p-2}} & \cdots & P_{W_{Z,p-1}}R_l(\mu)|_{W_{Z,1}}\\ 
        & \ddots & \ddots & \ddots & \vdots\\
        & & \ddots & \ddots &  P_{W_{Z,2}}R_l(\mu)|_{W_{Z,1}}\\
        & & & 0 & 0
    \end{bmatrix}
\end{equation}
by making use of the decomposition \eqref{eq:hilbert-space-decomposition}. We want to make use of this decomposition of $R_l(\mu)$ to analyse and represent solutions of \eqref{eq:inh-without-initial-value} with inhomogeneities on the whole space $Z$ and without needing further assumptions compared to \cite[Prop.~8.1]{gernandt_pseudo-resolvent_2023}. 
In the following, we focus on the decomposition of $Z$ in \eqref{eq:hilbert-space-decomposition} and rewrite $R_l(\mu)$. To do that, we write $W_k\coloneqq W_{Z,k}$, $k\in \{ 1,\ldots, p-1\}$. Of course the same procedure can be done for $R_r(\mu)$ on $X$.

By \eqref{eq:left-resolvent-decomposition} the DAE \eqref{eq:left-resolvent-dae} can be rewritten as
\begin{equation}\label{eq:left-resolvent-dae2}
    \frac{\dd}{\dd t} \begin{bsmallmatrix}
        R_l(\mu)|_{Z_{\ran}} & R_l(\mu)|_{W_{Z, p-1}} & P_{Z_{\ran}} R_l(\mu)|_{W_{Z,p-2}} & \cdots & P_{Z_{\ran}} R_l(\mu)|_{W_{Z,1}}\\
        0 & 0 & P_{W_{Z,p-1}} R_l(\mu)|_{W_{Z,p-2}} & \cdots & P_{W_{Z,p-1}}R_l(\mu)|_{W_{Z,1}}\\ 
        & \ddots & \ddots & \ddots & \vdots\\
        & & \ddots & \ddots &  P_{W_{Z,2}}R_l(\mu)|_{W_{Z,1}}\\
        & & & 0 & 0
    \end{bsmallmatrix} \begin{bmatrix}
        z_p \\ z_{p-1} \\ \vdots \\ z_2 \\ z_1 
    \end{bmatrix} = \begin{bmatrix}
        z_p \\ z_{p-1} \\ \vdots \\ z_2 \\ z_1 
    \end{bmatrix} + \begin{bmatrix}
        f_p \\ f_{p-1} \\ \vdots \\ f_2 \\ f_1 
    \end{bmatrix},
\end{equation}
as long as $f$ decomposes into $f_i \coloneqq P_{W_i}f$, $i\in \{ 1,\ldots,p-1\}$ and $f_p \coloneqq P_{Z_{\ran}}$.
In the proof of \cite[Prop.~8.3]{gernandt_pseudo-resolvent_2023} it was shown that one can solve $z_1,\ldots, z_{p-1}$ iteratively with $z_1 = -f_1$ and 
\begin{equation*}
    z_i = -f_i - \frac{\dd}{\dd t}\sum_{j=1}^{i-1} P_{W_{i-1}} R_l(\mu)\vert_{W_j} z_{i-1}, \quad i\in \{ 2,\ldots,p-1\},
\end{equation*}
if $f_i \in C^{p-i}([0,\infty); Z_{\ran})$ for $i\in \{1,\ldots,p-1\}$. Thus, \eqref{eq:left-resolvent-dae2} transforms into a DAE of the form
\begin{equation}\label{eq:DAE-after-decomposition}
    \frac{\dd}{\dd t} R_l(\mu)\vert_{Z_{\ran}} z_p = z_p + \hat f_p, \quad t\geq 0,
\end{equation}
where
\begin{equation*}
    \hat f_p = f_p + \frac{\dd}{\dd t} \sum_{j=1}^{p} P_{Z_{\ran}} R_l(\mu)\vert_{W_j} z_j.
\end{equation*}
Now, if $(R_l(\mu), I)$ fulfill the conditions of Proposition \ref{prop:E-on-X_ran-is-bd-inv}, one can find a $\tilde f_p$ such that $R_l(\mu) \tilde f_p = \hat f_p$.  
Thus, if 
\begin{equation}\label{eq:inh-for-hilber-space-case}
    \tilde v(t) = \int_0^t S_r(t-\tau) \tilde f_p(\tau)\dx[\tau], \quad t\geq 0,
\end{equation}
is $p+1$-times continuously differentiable, one can apply Lemma \ref{lem:solution-equals-derivative-of-v-part-2} to obtain a classical solution of \eqref{eq:left-resolvent-dae2}.
Note that the integrated semigroup $S_r$ inside the integral in \eqref{eq:inh-for-hilber-space-case} stands for the integrated semigroup generated by $(R_l(\mu), I)$, as we consider \eqref{eq:left-resolvent-dae2}. As a matter of fact, this coincides with the left-integrated semigroup generated by $(E,A)$.

If we now apply \cite[Lem.~8.1]{gernandt_pseudo-resolvent_2023} again, it is possible to transform the obtained solution of \eqref{eq:left-resolvent-dae2} back to a solution of \eqref{eq:inh-without-initial-value} and cover the case, where the inhomogeneity lies on the whole space $Z_{\ran}$.

\begin{ex}\hfill\\
    We return to the example from Section \ref{section:example}. Since $Z_{\ran} = \overline{\ran R_l(\mu)^1} = \overline{\ran R_l(\mu)^2} = L^2(0,1)\times \{0\}^3$ and $Z=Z_{\ran} \oplus W_1$ with $W_1= \{0\} \times L^2(0,1)\times \C \times \C$, one can decompose $R_l(\lambda)$ as seen above
    \begin{align*}
        R_l(\lambda) \begin{bmatrix}
            f \\ g \\ \mu_1 \\ \mu_2
        \end{bmatrix} 
        = R_l(\lambda)\vert_{Z_{\ran}} \begin{bmatrix}
            f \\ 0 \\ 0 \\ 0
        \end{bmatrix} + R_l(\lambda)\vert_{W_1}\begin{bmatrix}
            0 \\ g \\ \mu_1 \\ \mu_2
        \end{bmatrix} 
        = \begin{bmatrix}
            \int_0^\cdot\e^{-\lambda(\cdot-\sigma)}f(\sigma)\dx[\sigma] \\ 0 \\ 0 \\ 0
        \end{bmatrix} + \begin{bmatrix}
            \mu_1 \e^{-\lambda\cdot} \\ 0 \\ 0 \\ 0
        \end{bmatrix},
    \end{align*}
    for $\begin{bsmallmatrix}
        f & g & \mu_1 & \mu_2
    \end{bsmallmatrix}^\top\in Z$.
    Thus, for an inhomogeneity $h=\begin{bsmallmatrix}
        h_f & h_g & h_{\mu_1} & h_{\mu_2}
    \end{bsmallmatrix}^\top\in Z$ one can rewrite \eqref{eq:inh-without-initial-value} as it was done in \eqref{eq:DAE-after-decomposition}
    \begin{equation*}
        \frac{\dd}{\dd t} R_l(\lambda)\vert_{Z_{\ran}} \begin{bmatrix}
            f \\ 0 \\ 0 \\ 0
        \end{bmatrix} = \begin{bmatrix}
            f \\ 0 \\ 0 \\ 0
        \end{bmatrix} + \begin{bmatrix}
            h_f \\ 0 \\ 0 \\ 0
        \end{bmatrix} + \frac{\dd}{\dd t} R_l(\lambda)\vert_{W_1}\begin{bmatrix}
            0 \\ h_g \\ h_{\mu_1} \\ h_{\mu_2}
        \end{bmatrix} = \begin{bmatrix}
            f \\ 0 \\ 0 \\ 0
        \end{bmatrix} + \begin{bmatrix}
            h_f -\lambda h_{\mu_1} \e^{-\lambda\cdot} \\ 0 \\ 0 \\ 0
        \end{bmatrix},
    \end{equation*}
    or equivalently
    \begin{equation*}
        \frac{\dd}{\dd t} \int_0^t\e^{-\lambda(t-\sigma)}f(\sigma)\dx[\sigma] = f(t) + (h_f(t)-\lambda h_{\mu_1} \e^{-\lambda t}).
    \end{equation*}
\end{ex}

\section{Generation of a strongly continuous semigroup on a subspace} \label{section:semigroup-generation}

In this section, we show that the differential-algebraic integrated semigroups introduced in Section~\ref{section:solutions} generate a strongly continuous semigroup on a suitable subspace, analogous to \cite[Thm.~5.2]{neubrander_integrated_1988} for the abstract Cauchy problem. To address this, we extend the Assumption \ref{assumption:1} as follows.
\begin{assumption}\label{assumption:2}\hfill
    \begin{enumerate}[label=({\alph*)}]
        \item $X$ and $Z$ are Banach spaces.
        \item $E$ is a bounded linear operator from $X$ to $Z$.
        \item $A\colon\dom(A)\subseteq X\to Z$ is closed and densely defined.
        \item The operator pair $(E,A)$ has a \textit{complex resolvent index}, i.e.~the smallest number $p_{\mathrm{res}}^{(E,A)} \in \N_0$, such that there exists a $\omega \in\R$, $C>0$ with $\C_{\Re>\omega}\subseteq \rho(E,A)$ and 
        \begin{equation*}
            \Vert (\lambda E-A)^{-1} \Vert \leq C\vert\lambda\vert^{p_{\mathrm{res}}^{(E,A)}-1}\quad \text{for all }\lambda\in\C_{\Re>\omega}.
        \end{equation*}
        \item\label{assumption:2-e} The spaces $X_{\ran}$ and $\ker E$ intersect trivially, i.e.~$X_{\ran}\cap\ker E = \{0\}$.
    \end{enumerate}
\end{assumption}
The Assumption \ref{assumption:2}\ref{assumption:2-e} does not hold in general. It is, for example, satisfied in the presence of the radiality index \cite{erbay_index_2024, sviridyuk_linear_2003} or the resolvent growth condition $(D_k)$ introduced in \cite{gernandt_pseudo-resolvent_2023}, since these conditions imply $X_{\ran} \cap \ker R_r(\mu)^p = \{0\}$ and hence $X_{\ran} \cap \ker E = \{0\}$. 
Later, we focus on a specific class of DAEs which, when rewritten appropriately, always satisfy this condition.
Before stating the main result, we begin with the following necessary lemma and, in accordance with the notation of Section~\ref{section:solutions}, we set $p := p_{\mathrm{res}}^{(E,A)} + 1$.
\begin{lem}\label{lem:ker-and-X_ran-disjoint}\hfill\\
    Let $(E,A)$ satisfy Assumption \ref{assumption:2}, set $p:=p_{\mathrm{res}}^{(E,A)}+1$ and let $(S_l(t))_{t\geq 0}$ and $(S_r(t))_{t\geq 0}$  be the $p$-times integrated semigroups of $(E,A)$. Let $x\in X_{\ran}$. If $S_r(t)x=0$ for all $t\geq 0$, then $x=0$.  
\end{lem}

\begin{proof}
    By Lemma \ref{lem:properties-integrated-semigroup} \ref{lem:properties-integrated-semigroup-d} one obtains
    \begin{equation}\label{eq:lem-proof}
        0 = A\int_0^t S_r(\tau)x\dx[\tau] - ES_r(t)x = -\frac{t^p}{p!} Ex, \quad t\geq 0,
    \end{equation}
    and, especially, $Ex=0$.
    Since $X_{\ran} \cap \ker E = \{0\}$, the assertion follows.
\end{proof}

The next result is a generalization of \cite[Thm.~5.2]{neubrander_integrated_1988}. The main difficulty lies in omitting the fact that here $R_r(\mu)$ is a pseudo-resolvent, which can not be written as a resolvent on a subspace without further assumptions.
To ease the notation in the following, we denote for $k\in \N$ the $k$-th integral of $t\mapsto S_r(t)x$ for $x\in X_{\ran}$ by $S_r^{[k]}(t)x\coloneqq \int_0^t \frac{(t-v)^{k-1}}{(k-1)!} S_r(v)\dx[v]$. This exists due to Lemma \ref{lem:continuity-of-int-semigroup}. Further, we define $m\in \N_0$
\begin{equation}
    C_m\coloneqq \{ x\in X_{\ran} \; | \; t\mapsto S_r(t)x\in C^{m}([0,\infty); X_{\ran})\},\label{eq:Cmdef}
\end{equation}
where $C^{m}([0,\infty); X_{\ran})$ denotes the space of $m$-times continuously differentiable functions with values in $X_{\ran}$. We write $S_r^{(p)}(t)x$ to denote the $p$-th derivative of $t\mapsto S_r^{(p)}(t)x$ for $x\in C_p$.

\begin{thm}\label{thm:semigroup-on-a-subspace}\hfill\\
    Let $(E,A)$ satisfy Assumption \ref{assumption:2}, set $p:=p_{\mathrm{res}}^{(E,A)}+1$ and let $(S_l(t))_{t\geq 0}$ and $(S_r(t))_{t\geq 0}$  be the $p$-times integrated semigroups of $(E,A)$, with $\Vert S_r(t)\Vert\leq C\e^{\omega t}$ and $\Vert S_l(t)\Vert\leq C\e^{\omega t}$. 
    Then, for $C_m$ as in \eqref{eq:Cmdef},
    \begin{enumerate}[label=({\alph*)}]
        \item $C_{2p}\subseteq \ran R_r(\mu)^p \subseteq C_p$.
        \item $F\coloneqq \overline{C_{2p}}^{\Vert \cdot\Vert_F}\subseteq C_p$ is a Hilbert space, where
        \begin{equation*}
            \Vert x_0\Vert_{F}\coloneqq \sup_{t\geq 0} \Vert \e^{-\omega t} S_r^{(p)}(t)x_0\Vert
        \end{equation*}
        defines a norm on $\ran R_r(\mu)^p$.
        \item $(S_r^{(p)}(t))_{t\geq 0}$ is a strongly continuous semigroup on $F$. 
    \end{enumerate}
\end{thm}

\begin{proof}    
    Let $m\in \N_0$ and $x_0 \in C_m$. By Lemma \ref{lem:properties-integrated-semigroup} \ref{lem:properties-integrated-semigroup-f} one has
    \begin{equation}\label{eq:langer-beweis-1}
        S_r(t)S_r(s)x_0 = \int_0^t \frac{(t-v)^{p-1}}{(p-1)!} S_r(v+s)x_0 \dx[v] - \sum_{j=0}^{p-1} \frac{s^j}{j!}S_r^{[p-1-j]}(t)x_0.
    \end{equation}
    Thus, $S_r(s)x_0 \in C_{m+1}$ and consequently $S_r(s)C_m\subseteq C_{m+1}$. Let $k \in \N_0$ with $k\leq p-1$ and let $m\geq p$. Then, using \eqref{eq:langer-beweis-1} and the continuity of $S_r(t)$ from $X_{\ran}$ to $X_{\ran}$, one shows inductively
    \begin{align}\label{eq:thm-proof-1}
        \begin{split}
            S_r(t) S_r^{(k)}(s)x_0 & = \int_0^t \frac{(t-v)^{p-1-k}}{(p-1-k)!} S_r(\tau+s) x_0 \dx[\tau] - \sum_{j=0}^{k-1} \frac{t^{p-1-j}}{(p-1-j)!} S_r^{(k-1-j)}(s)x_0 \\
            & \quad -\sum_{j=k}^{p-1} \frac{s^{j-k}}{(j-k)!} S_r^{[p-1-j]}(t)x_0.
        \end{split}
    \end{align}
    Thus $S_r^{(k)}(s) C_m\subseteq C_{m+1}$ and $S_r(t)S_r^{(k)}(0) = 0$ for all $1\leq k \leq p-1$. Since $X_{\ran} \cap \ker E =\{0\}$, one can use Lemma \ref{lem:ker-and-X_ran-disjoint} to obtain $S_r^{(k)}(0)=0$ on $C_k$ for all $0\leq k \leq p-1$. Differentiating \eqref{eq:thm-proof-1} with respect to $s$ once more, one derives
    \begin{equation}\label{eq:thm-proof-2}
        S_r(t) S_r^{(p)}(s)x_0 = S_r(t+s)x_0 - \sum_{j=0}^{p-1} \frac{t^j}{j!} S_r^{(j)}(s)x_0
    \end{equation}
    and, thus, $S_r^{(p)}(s) C_m \subseteq C_m$ for all $p\leq m$. Rewriting \eqref{eq:thm-proof-2}, one obtains
    \begin{equation*}
        S_r(t) \left(  S_r^{(p)}(0) - I_{C_p} \right) = 0
    \end{equation*}
    and $\frac{\dd^p}{\dd s^p} S_r(0) = I_{C_p}$ on $C_p$ by Lemma \ref{lem:ker-and-X_ran-disjoint}, where $I_{C_p}$ denotes the identity on $C_p$. Taking the $p$-th derivative with respect to $t$ from \eqref{eq:thm-proof-2} one obtains
    \begin{equation*}
        S_r^{(p)}(t) S_r^{(p)}(s) = S_r^{(p)}(t+s)
    \end{equation*}
    on $C_m$, $m\geq p$. Thus, $(S_r(t))_{t\geq 0}$ is a strongly continuous semigroup on $C_m$ for all $m\geq p$.
    Now, let $x_0 \in C_m$, $m\geq p+1$, and take the derivative of \eqref{eq:thm-proof-2} along $s$
    \begin{equation*}
        S_r(t) S_r^{(p+1)}(s)x_0 = S_r^{(1)}(t+s)x_0 - \sum_{j=0}^{p-1}\frac{t^j}{j!} S_r^{(j+1)}(s)x_0.
    \end{equation*}
    Thus, $S_r^{(p+1)}(s)C_m\subseteq C_{m-1}$, for all $m\geq p+1$, and 
    \begin{equation*}
        S_r(t)S_r^{(p+1)}(0)x_0 = S_r^{(1)}(t)x_0 - \frac{t^{p-1}}{(p-1)!} x_0,
    \end{equation*}
    as $S_r^{(p)}(0)=I_{C_p}$. Hence, by integration by parts,
    \begin{align*}
        R_r(\mu) S_r^{(p+1)}(0)x_0 &= \mu^p \int_0^\infty \e^{-\mu t} S_r(t) S_r^{(p+1)}(0)x_0 \dx[t] \\
        &= \mu^p \int_0^\infty \e^{-\mu t}S_r^{(1)}(t)x_0 - \frac{t^{p-1}}{(p-1)!} x_0 \dx[t]\\
        &= \mu R_r(\mu)x_0 - x_0.
    \end{align*}
Consequently, one obtains
    \begin{equation*}
        R_r(\mu) \left( \mu I_{C_m} - S_r^{(p+1)}(0)\right) = I_{C_m}
    \end{equation*}
    on $C_m$ and $C_m\subseteq R_r(\mu) C_{m-1}$ for $m\geq p+1$. Since $S_r(t)C_m\subseteq C_{m+1}$ and $R_r(\mu) = \mu^p \int_0^\infty \e^{-\mu t} S_r(t)\dx[t]$ one has $R_r(\mu)C_m \subseteq C_{m+1}$ for all $m\geq 0$.
    Combining these two, one obtains
    \begin{equation}\label{eq:resolvent-on-C_m}
        R_r(\mu) C_{m-1} = C_m, \quad m\geq p+1
    \end{equation}
    and $C_{2p} = R_r(\mu) C_{2p-1} = \ldots = R_r(\mu)^p C_p \subseteq \ran R_r(\mu)^p\vert_{X_{\ran}}$.
    By Theorem \ref{thm:solutions-complex-resolvent-index} there exists for all $x_0\in \ran R_r(\mu)^p$ a $z_0\in X$ with $x_0=R_r(\mu)^p z_0$ and a mild solution $x(\cdot)$ with
    \begin{equation}\label{eq:estimate-2}
        \Vert S_r^{(p)}(t)x_0\Vert =\Vert x(t)\Vert \leq C \e^{\omega t} \Vert z_0\Vert, \quad t\geq 0,
    \end{equation}
    and, in particular, $\ran R_r(\mu)^p\subseteq C_p$ holds, which shows the first assertion. Note that we used the fact that $S_r(\cdot)x_0$ is the $p$-th integral of the mild solution $x(\cdot)$.
    
    To address the second assertion, we define on $\ran R_r(\mu)^p$ the norm 
    \begin{equation*}
        \Vert x_0\Vert_F\coloneqq \sup_{t\geq 0} \Vert \e^{-\omega t} S_r^{(p)}(t) x_0\Vert
    \end{equation*}
    and denote the closure of $C_{2p}$ with respect to $\Vert \cdot\Vert$ by $F$. By invoking $S_r^{(p)}(0)x_0 = x_0$, we have\begin{equation}\label{eq:norm-estimate}
        \Vert x\Vert \leq \Vert x\Vert_F.
    \end{equation}
    Let $(x_n)_n\subseteq C_{2p}$ with $x_n \to x\in F$ with respect to $\Vert \cdot \Vert_F$. By \eqref{eq:norm-estimate} $x_n$ converges to $x$ in $X_{\ran}$ and, since $S_r(t)$ is bounded, $S_r(t)x_n$ converges to $S_r(t)x$. Since $x_n\in C_{2p}$ the functions $t\mapsto S_r(t)x_n$ are at least $p$-times continuously differentiable. Using the semigroup property of $(S_r^{(p)}(t))_{t\geq 0}$ on $C_{2p}$, one has
    \begin{equation}\label{eq:estimate-3}
        \Vert S_r^{(p)}(s)x_n\Vert_F = \sup_{t\geq 0} \Vert \e^{-\omega t} S_r^{(p)}(t+s)x_n \Vert \leq \e^{\omega s} \Vert x_n\Vert_F.
    \end{equation}
    Together with \eqref{eq:norm-estimate} one obtains
    \begin{equation}\label{eq:cauchy-sequence-estimate}
        \Vert S_r^{(p)}(s)x_n - S_r^{(p)}(s)x_m\Vert\leq \e^{\omega s}\Vert x_n -x_m\Vert_F,
    \end{equation}
    and since $S_r^{(k)}(0)x_n=0$ for all $1\leq k \leq p-1$ and $n\in \N$, $t\mapsto S_r(t)x$ is $p$-times continuously differentiable. Thus, $x\in C_p$ and $F\subseteq C_p$.

    Now, let $x\in F\subseteq C_p$. Then $S_r^{(p)}(t)x$ is well-defined. Let $(x_n)_n\subseteq C_{2p}$ be a sequence converging to $x$ with respect to $\Vert \cdot\Vert_F$. Then by \eqref{eq:cauchy-sequence-estimate} and \eqref{eq:norm-estimate} $(S_r^{(p)}(t)x_n)_n$ is a Cauchy sequence in $F$ and in $X_{\ran}$. Thus, $S_r^{(p)}(t)x_n$ converges to some $g \in F$ in $\Vert \cdot\Vert_F$ and $S_r^{(p)}(t)x$ in $X_{\ran}$. Thus, $S_r^{(p)}(t)x = g\in F$. Consequently, $S_r^{(p)}(t)\colon F\to F$ is well-defined for all $t\geq 0$. Since $(S_r^{(p)}(t))_{t\geq 0}$ is a semigroup on $C_{p}$, it is easy to see that it is a semigroup on $F$ as well, as $F$ is the closure of $C_{2p}(\subseteq C_p)$. 
    
    What is left to show is that $(S_r^{(p)}(t))_{t\geq 0}$ is strongly continuous. 
    Let $x\in C_{2p}$. Since $C_{2p}\subseteq \ran R_r(\mu)^p\vert_{C_p}$ there exist $z\in C_p$ with $x=R_r(\mu)^pz$. Similar to Lemma \ref{lem:properties-integrated-semigroup} \ref{lem:properties-integrated-semigroup-a} one can show that $S_r^{(p)}(t) R_r(\mu)^p z = R_r(\mu)^p S_r^{(p)}(t)z$ holds. Then \eqref{eq:estimate-2} leads to
    \begin{equation*}
        \Vert R_r(\mu)^p z\Vert_F = \sup_{t\geq 0} \Vert \e^{-\omega t} S_r^{(p)}(t)R_r(\mu)^p z_0 \Vert \leq C \Vert z_0 \Vert,
    \end{equation*}
    and thus, by using that $(S_r^{(p)}(t))_{t\geq 0}$ is a strongly continuous semigroup on $C_p$,
    \begin{align*}
        \Vert S_r^{(p)}(t) x_0 - x_0 \Vert_F &= \Vert S_r^{(p)}(t) R_r(\mu)^p z_0 - R_r(\mu)^p z_0 \Vert_F\\
        &= \Vert R_r(\mu)^p \left( S_r^{(p)}(t)z_0 - z_0\right)\Vert_F\\
        &\leq \Vert S_r^{(p)}(t)z_0 - z_0 \Vert_F\\
        &\to 0, \qquad t\to 0.
    \end{align*}
     Since $C_{2p}$ is dense in $F$, it follows from \eqref{eq:estimate-3} that $\Vert S_r^{(p)}(t) x-x\Vert_F\to 0$ for $t\to 0$ holds for all $x\in F$. Thus, $(S_r^{(p)}(t))_{t\geq 0}$ is strongly continuous on $F$.
\end{proof}

Finally, let us place some emphasis on the class of \textit{abstract dissipative Hamiltonian DAE}
\begin{equation}\label{eq:adH-DAE-1}
    \tfrac{\dd}{\dd t} Ex(t) = DQx(t), \quad t\geq 0.
\end{equation}
Hereby, $E, Q\in L(X,Z)$, such that $Q$ is boundedly invertible, $E$ has  closed range, and $E^\ast Q = Q^\ast E$ is self-adjoint and nonnegative. Further, $A\colon \dom(A)\subset Z\to Z$ is closed, densely defined and maximally dissipative. Note that the finite-dimensional counterpart of these equations has been treated in \cite{beattie_linear_2018}.


Using that $Q$ is boundedly invertible and $E^\ast Q = Q^\ast E$ is nonnegative, it follows that $EQ\inv = Q\ainv E^\ast$ is nonnegative \cite[Rem.~4.1 b)]{erbay_jacob_morris24}. 
Thus, by a~multiplication of \eqref{eq:adH-DAE-1} from the left with $Q\inv$, we obtain 
an~abstract dissipative Hamiltonian DAE with $X=Z$, $Q=I_Z$ and $E$ being self-adjoint and nonnegative. 
As done in in Corollary \ref{cor:1}, we may focus on the (not necessarily abstract dissipative Hamiltonian) DAE 
\begin{equation}\label{eq:DAE-left-resolvent}
    \frac{\dd }{\dd t} \underbrace{E(\mu E-A)\inv}_{\eqqcolon \tilde E} z(t) = \underbrace{A(\mu E-A)\inv }_{\eqqcolon \tilde A} z(t), \qquad t\geq 0.
\end{equation}
Then the left- and right-integrated semigroups generated by $(\tilde E,\tilde A)$, 
denoted by $(\tilde S_r(t))_{t\geq 0}$ and $(\tilde S_l(t))_{t\geq 0}$, 
both coincide with the left-integrated semigroup generated by $(E,A)$. Thus we may write
\begin{equation*}
    (\tilde S(t))_{t\geq 0} \coloneqq (\tilde S_r(t))_{t\geq 0} 
      = (\tilde S_l(t))_{t\geq 0} = (S_l(t))_{t\geq 0}.
\end{equation*}
Moreover, the spaces $X_{\ran}$ and $Z_{\ran}$ associated with $(\tilde E,\tilde A)$ 
both agree with $Z_{\ran}$ associated with $(E,A)$.
 By \cite[Prop.~7.3\&Thm.~5.1]{gernandt_pseudo-resolvent_2023} we have
\begin{equation}\label{eq:Z_ran-and-Z_ker-disjoint}
    Z_{\ran} \cap \ker R_l(\mu)^p = \{0\}
\end{equation}
and, in particular $Z_{\ran} \cap \ker \tilde E=\{0\}$. Thus it is possible to apply Theorem \ref{thm:semigroup-on-a-subspace} to $(\tilde E, \tilde A)$ and obtain a strongly continuous semigroup on a subspace. 
Further, in \cite[Thm.~5.1]{gernandt_pseudo-resolvent_2023} it was shown that there exist an operator $A_R\colon \dom(A_R)\subseteq Z_{\ran} \to Z_{\ran}$ with $\dom(A_r)=R_l(\mu)(Z_{\ran})$ and
\begin{equation*}
    (\lambda I -A_R)\inv = R_l(\mu)\vert_{Z_{\ran}}.
\end{equation*}
Furthermore, it was shown in \cite[Prop.~8.1]{gernandt_pseudo-resolvent_2023} that if $A_R$ generates a $C_0$-semigroup, then \eqref{eq:DAE-left-resolvent} has weak solutions on the whole subspace $Z_{\ran}$.
In the setting of Theorem~\ref{thm:semigroup-on-a-subspace}, the operator $A_R$ is 
precisely the generator of $(S_r^{(p)}(t))_{t\geq 0}$ and coincides with
\begin{equation*}
    S_r^{(p+1)}(0)\colon C_{p+1}\subseteq F \to F
\end{equation*}
on $C_{p+1}$. Thus, Theorem~\ref{thm:semigroup-on-a-subspace} highlights the role 
of $A_R$ in cases where it fails to generate a strongly continuous semigroup. 

\section*{Acknowledgement}
The authors would like to thank Hannes Gernandt, who drew our attention to the Hilbert space decomposition which is used in Section \ref{Section:hilbert-space-decomp}.\\
The authors gratefully acknowledge funding from the Deutsche Forschungsgemeinschaft (DFG, German Research Foundation), Project-ID 531152215,
CRC 1701 “Port-Hamiltonian Systems”.





\bibliographystyle{abbrv}
\bibliography{references}

\end{document}